\documentclass[a4paper, 12pt]{article}

\usepackage{amssymb}
\usepackage{bm}
\usepackage{textcomp}
\usepackage{mathrsfs}
\usepackage{amsmath}
\usepackage{amsfonts}
\usepackage[T1]{fontenc}
\usepackage[latin9]{inputenc}
\usepackage{amsthm}
\usepackage{graphicx}
\usepackage{amsmath}

\usepackage{amsthm}
\usepackage{array}
\usepackage{cases}
\usepackage{enumerate}
\usepackage{clrscode}
\usepackage{listings}
\usepackage[ruled,vlined,english,linesnumbered]{algorithm2e}
\usepackage{url}

\usepackage{enumerate,enumitem,todonotes}
\makeatletter
\def\th@plain{%
  \itshape 
}
\makeatother

\makeatletter
\renewenvironment{proof}[1][\proofname]{\par
  \pushQED{\qed}%
  \normalfont \topsep6\p@\@plus6\p@\relax
  \trivlist
  \item[\hskip\labelsep
        \bfseries
    #1\@addpunct{.}]\ignorespaces
}{%
  \popQED\endtrivlist\@endpefalse
}
\makeatother

\newtheorem{theorem}{Theorem}[section]

\numberwithin{equation}{section}

\newtheorem{lemma}[theorem]{Lemma}
\newtheorem{thm}{Theorem}[section]

\newtheorem{conj}[thm]{Conjecture}
\newtheorem{pblm}[thm]{Problem}

\numberwithin{equation}{section}

\usepackage[varg]{txfonts}
\usepackage{graphicx, epsfig, subfigure}
\usepackage{enumerate} 
\usepackage[square, numbers, sort&compress]{natbib} 
\ifx\pdfoutput\undefined
 \usepackage[dvipdfm,%
  pdfstartview=FitH, 
  bookmarks=true,%
  bookmarksnumbered=true, 
  bookmarksopen=true, 
  plainpages=false,%
  pdfpagelabels,%
  colorlinks=true, 
  linkcolor=blue, 
  citecolor=blue,%
  urlcolor=black,
  pdfborder=001]{hyperref}
  \AtBeginDvi{}  
\else
 \usepackage[pdftex,%
  pdfstartview=FitH, 
  bookmarks=true,%
  bookmarksnumbered=true, 
  bookmarksopen=true, 
  plainpages=false,%
  pdfpagelabels,%
  colorlinks=true, 
  linkcolor=blue, 
  citecolor=blue,%
  urlcolor=black,
  pdfborder=001]{hyperref}
\fi

\numberwithin{equation}{section}

\begin{document}

\title{\LARGE Proper conflict-free choosability of planar graphs}
\author{Yuting Wang~~~~~~~~ Xin Zhang\thanks{Corresponding author. Email: xzhang@xidian.edu.cn.}\\
{\small School of Mathematics and Statistics, Xidian University, Xi'an, 710071, China}}


\maketitle

\begin{abstract}\baselineskip 0.60cm
A proper conflict-free coloring of a graph is a proper vertex coloring wherein each non-isolated vertex's open neighborhood contains at least one color appearing exactly once. For a non-negative integer $k$, a graph $G$ is said to be proper conflict-free (degree+$k$)-choosable if  
given any list assignment $L$ for $G$ where $|L(v)| = d(v) + k$ holds for every vertex $v \in V(G)$,  
there exists a proper conflict-free coloring $\phi$ of $G$ such that $\phi(v) \in L(v)$ for all $v \in V(G)$.
Recently, Kashima, {\v{S}}krekovski, and Xu proposed two related conjectures on proper conflict-free choosability: the first asserts the existence of an absolute constant $k$ such that every graph is proper conflict-free \textnormal{(degree+$k$)}-choosable, while the second strengthens this claim by restricting to connected graphs other than the cycle of length 5 and reducing the constant to $k=2$. In this paper, we confirm the second conjecture for three graph classes: $K_4$-minor-free graphs with maximum degree at most 4, outer-1-planar graphs with maximum degree at most 4, and planar graphs with girth at least 12; we also confirm the first conjecture for these same graph classes, in addition to all outer-1-planar graphs (without degree constraints). Moreover, we prove that planar graphs with girth at least 12 and  outer-1-planar graphs are proper conflict-free $6$-choosable.

\vspace{3mm}\noindent \emph{Keywords: proper conflict-free coloring; planar graph; $K_4$-minor-free graph; outer-$1$-planar graph}.
\end{abstract}

\baselineskip 0.60cm

\section{Introduction}\label{sec1}


In this paper, we consider only simple, finite graphs. For a graph $ G $, let $ V(G) $ and $ E(G) $ denote its vertex and edge sets, respectively.  
The \emph{open neighborhood} $ N_G(v) $ of a vertex $ v \in V(G) $ is the set of vertices adjacent to $ v $. The \emph{degree} $ d_G(v) $ of $ v $ is the number of vertices in its open neighborhood.  
The \emph{maximum degree} of $ G $, denoted by $ \Delta(G) $, is the largest degree among all vertices in $ G $. 
If the graph $ G $ is planar, we assume it is embedded in the plane so that no edges cross. We denote the set of faces of this embedding by $ F(G) $. The \emph{degree} $ d_G(f) $ of a face $ f \in F(G) $ is the number of edges incident to $ f $, where each cut-edge is counted twice. When the graph $ G $ is clear from context, we simply write $ N(v) $, $ d(v) $, and $ d(f) $ for brevity.

A \emph{proper $k$-coloring} of a graph $G$ is a mapping $\phi\colon V(G) \to [k] := \{1, 2, \dots, k\}$ such that $\phi(u) \neq \phi(v)$ for every edge $uv \in E(G)$.
An \emph{$f$-list assignment} for $G$ (with $f\colon V(G) \to \mathbb{N}^+$) is a function $L\colon V(G) \to 2^{\mathbb{N}^+}$ that assigns to each vertex $v$ a set $L(v)$ of at least $f(v)$ colors. Given an $ f $-list assignment $ L $, an \textit{$ L $-coloring} of $ G $ is a proper coloring $ \phi $ satisfying $ \phi(v) \in L(v) $ for all $ v \in V(G) $. If $ |L(v)| = k $ for each vertex $ v $, then this $ L $-coloring is called a \textit{list $ k $-coloring}.


In 2023, Fabrici, Lu{\v{z}}ar, Rindo{\v{s}}ov\'a, and Sot\'ak \cite{FABRICI202380} introduced the notion of proper conflict-free coloring for graphs. 
A \emph{proper conflict-free $k$-coloring} of $G$ is a proper coloring $\phi\colon V(G) \to [k]$ such that for every non-isolated vertex $v \in V(G)$, the set 
\[
U_\phi(v,G) := \big\{ \phi(u) \mid u \in N(v) \text{ and } \phi(u) \neq \phi(w) \text{ for all } w \in N(v) \setminus \{u\} \big\}
\]
is non-empty. In other words, $U_\phi(v,G)$ denotes the set of colors that appear exactly once in the open neighborhood $N(v)$ of $v$. The minimum $k$ such that $G$ admits a proper conflict-free $k$-coloring is the \textit{proper conflict-free ({\rm PCF}) chromatic number} of $G$, denoted by $\chi_{\rm PCF}(G)$. 
Caro, Petru{\v{s}}evski, and {\v{S}}krekovski \cite{CARO2023113221} proved $\chi_{\rm PCF}(G) \leq \lfloor 5\Delta(G) / 2 \rfloor$ for every graph $G$ with maximum degree $\Delta(G) \geq 1$ and conjectured the following:
\begin{conj}\label{conj-A}
\cite{CARO2023113221}
For any connected graph $G$ with maximum degree $\Delta\geq 3$, the proper conflict-free chromatic number satisfies:
\[
\chi_{\rm PCF}(G) \leq \Delta + 1.
\]
\end{conj}

The special case of the $5$-cycle $ C_5 $, for which $ \chi_{\rm PCF}(C_5) = 5 \neq \Delta + 1 $, demonstrates that the condition $ \Delta \geq 3 $ cannot be relaxed.  
Moreover, for any cycle other than $ C_5 $ whose length is not a multiple of $3$, we have $\chi_{\rm PCF}(C) = 4 \neq \Delta + 1 $, leading to a similar conclusion as for \( C_5 \).
The upper bound $ \Delta + 1 $ is tight, as evidenced by the fact that $ \chi_{\rm PCF}(G) = \Delta(G) + 1 $ when $ G $ is the \textit{$1$-subdivision} (obtained by subdividing each edge exactly once) of the complete graph $ K_n $; see \cite{arXiv:2505.04543}.

An earlier work on the linear coloring of graphs by Liu and Yu \cite{zbMATH06282718}  implies Conjecture \ref{conj-A} in the case of $\Delta(G) = 3$; their results also cover the list version of the conjecture. As next steps toward resolving Conjecture \ref{conj-A}, Cranston and Liu \cite{zbMATH07959405} demonstrated that a straightforward greedy algorithm yields a proper conflict-free $(2\Delta + 1)$-coloring for any graph $ G $ with maximum degree $ \Delta $. This bound was later improved to $(2\Delta - 1)$ by Cho, Choi, Kwon, and Park \cite{CHO2025114233}, which currently stands as the best-known upper bound without assuming $ \Delta $ is sufficiently large.
On the other hand, if $ \Delta $ is sufficiently large, then Cranston and Liu \cite{zbMATH07959405} proved $\chi_{\rm PCF}(G)  \leq 1.656\Delta$, Liu and Reed \cite{zbMATH08036330} proved 
\[
    \chi_{\rm PCF}(G)  \leq \Delta + \mathcal{O}\left(\Delta^{2/3} \log \Delta\right),
\]
showing that Conjecture \ref{conj-A} holds asymptotically,
and Chuet, Dai, Ouyang, and Pirot \cite{arXiv:2505.04543} refined this upper bound by removing the polynomial factor in the second-order term, giving
\[
    \chi_{\rm PCF}(G)  \leq \Delta + \mathcal{O}\left( \log \Delta\right).
\]
We refer the reader to \cite{CARO2023113221, zbMATH07585605, arXiv:2508.20521, arXiv:2509.12560, LIU2024113668, CHO2025114233, KASHIMA2026114800, FABRICI202380, arXiv:2505.04543, CHO202534, zbMATH08036330, zbMATH07959405} for other results concerning proper conflict-free colorings of graphs.

Research on the list version of proper conflict-free coloring has been conducted in \cite{LIU2024113668, zbMATH07959405}; in particular, Liu \cite{LIU2024113668} proved that every graph with layered treewidth at most $ w $ is proper conflict-free $(8w - 1)$-choosable, while Liu and Cranston \cite{zbMATH07959405} demonstrated that every graph with sufficiently large maximum degree $ \Delta $ is proper conflict-free $ 1.656\Delta $-choosable. 
As an analogy to the \emph{degree-choosability} of graphs, Kashima, \v{S}krekovski, and Xu \cite{arXiv:2508.20521} introduced the notion of \emph{proper conflict-free (degree + $k$)-choosability}. For a non-negative integer $k$, a graph $G$ is said to be \emph{proper conflict-free $({\rm degree} + k)$-choosable} if, for any $f$-list assignment $L$ of $G$ satisfying $f(v) = d(v) + k$ for all $v \in V(G)$, there exists a proper conflict-free $L$-coloring $\phi$ of $G$. Here, the condition $f(v) = d(v) + k$ ensures that each vertex's list size exceeds its degree by a fixed margin $k$. Kashima, {\v{S}}krekovski, and Xu proposed the following conjecture in their paper \cite{arXiv:2508.20521}.

\begin{conj}\label{conj-B}
\cite{arXiv:2508.20521}
There exists an absolute constant $k$ such that every graph is proper conflict-free {\rm (degree+$k$)}-choosable. 
\end{conj}

Note that the $5$-cycle does not satisfy the property of proper conflict-free 4-choosability, which suggests that the minimal value of $k$ must be at least $3$. Furthermore, Kashima, \v{S}krekovski, and Xu \cite{arXiv:2508.20521} conjectured that the $5$-cycle is the only graph that fails to be proper conflict-free (degree+$2$)-choosable. They formally stated their stronger conjecture as follows:

\begin{conj}\label{conj-C}
\cite{arXiv:2508.20521}
Every connected graph other than $C_5$ is proper conflict-free {\rm (degree+$2$)}-choosable. 
\end{conj}

Since cycles of length $l \not\equiv 0 \pmod{3}$ are known to fail proper conflict-free {\rm (degree+$1$)}-choosable \cite{CARO2023113221}, the upper bound cannot be further improved to guarantee proper conflict-free {\rm (degree+$1$)}-choosable for all graphs.
On the positive side, Kashima, {\v{S}}krekovski, and Xu  \cite{arXiv:2508.20521} confirmed Conjecture \ref{conj-B} for graphs with maximum degree at most $4$ and Conjecture \ref{conj-C} for subcubic graphs; in particular, they proved the following:

\begin{thm} \label{thm:cubic}
\cite{arXiv:2508.20521}
~~~
    \begin{enumerate}[label=$(\arabic*)$]\setlength{\itemsep}{-3pt}
        \item Every connected graph with maximum degree at most $4$ is proper conflict-free {\rm (degree+$3$)}-choosable.
        \item \label{two} Every connected subcubic graph other than $C_5$ is proper conflict-free {\rm (degree+$2$)}-choosable. 
    \end{enumerate}
\end{thm}

A graph $G$ is \textit{$d$-degenerate} if every subgraph of $G$ contains a vertex of degree at most $d$.  
Kashima, {\v{S}}krekovski, and Xu \cite{arXiv:2509.12560} proved that every $d$-degenerate graph is proper conflict-free {\rm (degree+$d$+1)}-choosable. Since the 5-cycle is $2$-degenerate but fails proper conflict-free {\rm (degree+$2$)}-choosable, the upper bound {\rm (degree+$d$+1)} cannot be uniformly reduced. However, trees achieve proper conflict-free {\rm (degree+1)}-choosable \cite{arXiv:2509.12560}, improving the bound by 1 compared to the general case.
Moreover, Kashima, {\v{S}}krekovski, and Xu \cite{KASHIMA2026114800} established the following results for outerplanar graphs --- a well-known subclass of $2$-degenerate planar graph:  

\begin{thm}\label{thm-OP}
\cite{KASHIMA2026114800}
Every connected outerplanar graph other than $C_5$ is proper conflict-free {\rm (degree+$2$)}-choosable. 
\end{thm}

This improves the general bound {\rm (degree+$d$+$1$)} by $1$ when $d=2$ and $G$ is an outerplanar graph. In the literature, there are two subclasses of planar graphs that share some similarities with outerplanar graphs: 
\begin{itemize}
    \item \textit{$K_4$-Minor-free Graphs} \cite{zbMATH08086187,zbMATH08068282,zbMATH07888732}: A graph $G$ is $K_4$-minor-free if it contains no subgraph obtainable from $K_4$ via edge/vertex deletions or contractions.
    \item \textit{Outer-$1$-Planar Graphs} \cite{zbMATH06587983,zbMATH06477666,zbMATH06081325}: A graph $G$ is outer-$1$-planar if it admits a drawing with all vertices on the outer face and at most one crossing per edge.
\end{itemize}
It is known that $K_4$-minor-free graphs are $2$-degenerate and outer-$1$-planar graphs are $3$-degenerate. Thus:
\begin{itemize}
    \item Every $K_4$-minor-free graph is proper conflict-free {\rm (degree+$3$)}-choosable.
    \item Every outer-1-planar graph is proper conflict-free {\rm (degree+$4$)}-choosable.
\end{itemize}

In this paper, we improve these bounds to {\rm (degree+$2$)} for graphs with maximum degree $\Delta \leq 4$ by establishing the following theorem:
\begin{thm}\label{thm-SP}
Every connected $K_4$-minor-free graph with maximum degree at most $4$ other than $C_5$ is proper conflict-free {\rm (degree+$2$)}-choosable. 
\end{thm}

\begin{thm}\label{thm-O1P}
Every connected outer-$1$-planar graph with maximum degree at most $4$ other than $C_5$ is proper conflict-free {\rm (degree+$2$)}-choosable. 
\end{thm}

\noindent Besides, we remove the restriction on maximum degree and prove the following results for planar graphs with large girth and outer-$1$-planar graphs:

\begin{thm}\label{thm-g}
Every planar graph with girth at least $12$ is proper conflict-free {\rm (degree+$2$)}-choosable. 
\end{thm}

\begin{thm}\label{thm-O1P-degree+3}
Every outer-$1$-planar graph is proper conflict-free {\rm (degree+$3$)}-choosable, and this bound is sharp.
\end{thm}

\noindent Therefore, we confirm Conjecture \ref{conj-B} for outer-$1$-planar graphs and Conjecture \ref{conj-C} for 
subquartic $K_4$-minor-free graphs, subquartic outer-$1$-planar graphs, and planar graphs with girth at least $12$. This resolves both conjectures in multiple restricted settings and significantly extends their applicability.

\section{Unavailable configurations}

We prove in this section that every graph under discussion admits at least one configuration from the subsequent collection:

\begin{enumerate}[label=$T_{\arabic*}$]\setlength{\itemsep}{-3pt}
     \item \label{t1} a $1$-vertex $u$;
     \item \label{t2} a triangle $uvw$ with $d(u) =2$ and $d(v) =3$;
     \item \label{t3} a 4-cycle $xuyv$ with $d(u)= d(v) =2$; 
     \item \label{t4} two incident triangles $uvw$ and $uxy$ such that $d(u)=4$ and $d(v)=d(x)=2$;
     \item \label{t5} 
     a path $uvw$ with $d(u)=d(v)=2$ and $d(w)=3$;
     \item \label{t6} 
     a triangle $uvx$ with $d(u)=d(v)=2$ and $d(x)\neq 3$; 
     \item \label{t7} a 4-cycle $uvwx$ with $d(u) = d(v) = d(w) =2$ and $d(x)\neq 3$;
     \item \label{t8} a trial $xuvwpy$ with $d(u) = d(v) = d(w) = d(p) =2$ such that $d(x)\neq 2$ whenever $x=y$;  
     \item \label{t9} 
     a triangle $ypw$ incident with a path $wvu$ such that $d(p)=d(v)=d(u)=2$ and $d(w)=4$;
     \item \label{t10} 
     a $4$-cycle $uvwx$ incident with a path $wpq$ such that $d(u)=d(v)=d(p)=d(q)=2$ and $d(w)=4$;
     
     \item \label{t11} 
     a path $xuvwpq$ such that $d(x)=d(u)=d(v)=d(p)=d(q)=2$ and $d(w)=4$;
     \item \label{t12} a path $uvwp$ incident with an edge $wa$ such that $d(u)=d(v)=d(p)=d(a)=2$ and $d(w)=4$;
     \item \label{t13} a triangle $xuw$ adjacent to a quadrilateral $xwvy$ such that $d(x)=d(w)=4$ and $d(u)=d(v)=2$; 
     \item \label{t14} a path $xuvy$ such that $d(u)=d(v)=3$ and $uy,xv \in E(G)$;
     \item \label{t15} a path $xwuvy$ such that $d(w)=2$, $d(u)=d(v)=3$, and $uy,xv \in E(G)$; 
     \item \label{t16} a path $xutvy$ such that $d(t)=2$, $d(u)=d(v)=3$, and $uy,xv \in E(G)$; 
     \item \label{t17} a path $xwuvzy$ such that $d(w)=d(z)=2$, $d(u)=d(v)=3$, and $uy,xv \in E(G)$; 
     \item \label{t18} a path $xwutvy$ such that $d(w)=d(t)=2$, $d(u)=d(v)=3$, and $uy,xv \in E(G)$; 
     \item \label{t19} a path $xwutvzy$ such that $d(w)=d(t)=d(z)=2$, $d(u)=d(v)=3$, and $uy,xv \in E(G)$; 
     \item \label{t20} two adjacent quadrilaterals $xuvw$ and $xwpy$ such that $d(u)=d(v)=d(p)=2$ and $d(x)=d(w)=4$;
     \item \label{t21} a $5$-cycle $xuvwy$ adjacent to a $4$-cycle $xypq$ such that $d(u)=d(v)=d(w)=d(p)=2$ and $d(x)=d(y)=4$; 
     \item \label{t22} a triangle $uxy$ with $d(u)=2$, $d(x)=d(y)=4$, and $N(x)\cap N(y)=\{u,h,p\}$;
     \item \label{t23} two incident triangles $xuy$ and $yph$ such that $d(u)=2$, $d(x)=d(y)=4$, 
     and $xh \in E(G)$; 
     \item \label{t24} a triangle $xuy$ incident with a quadrilateral $ypzh$ such that $d(u)=d(z)=2$, $d(h)=3$, $d(x)=d(y)=4$, and $xh \in E(G)$;
     \item \label{t25} a quadrilateral $xuvy$ such that $d(u)=d(v)=2$, $d(x)=d(y)=4$, and $N(x) \cap N(y) =\{h,p\}$; 
     \item \label{t26} a quadrilateral $xuvy$ incident with a triangle $yph$ such that $d(u)=d(v)=2$, $d(x)=d(y)=4$, and $xh \in E(G)$;
     \item \label{t27} two incident quadrilaterals $xuvy$ and $ypzh$ such that $d(u)=d(v)=d(z)=2$, $d(h)=3$, $d(x)=d(y)=4$, and $xh \in E(G)$;
     \item \label{t28} a $5$-cycle $xuwvy$ such that $d(u)=d(v)=d(w)=2$, $d(x)=d(y)=4$, and $N(x) \cap N(y)=\{h,p\}$; 
     \item \label{t29} a $5$-cycle $xuwvy$ incident with a triangle $ypq$ such that $d(u)=d(v)=d(w)=2$, $d(x)=d(y)=4$, and $xq \in E(G)$;
     \item \label{t30} a $5$-cycle $xuwvy$ incident with a quadrilateral $ypzq$ such that $d(u)=d(v)=d(w)=d(z)=2$, $d(q)=3$, $d(x)=d(y)=4$, and $xq\in E(G)$;
     \item \label{t31} a $5$-cycle $xuvyt$ such that $d(u)=d(v)=d(t)=2$, $d(x)=d(y)=4$, and $N(x) \cap N(y)=\{h,p,t\}$;
     \item \label{t32} a $5$-cycle $xuvyt$ incident with a triangle $yph$ such that $d(u)=d(v)=d(t)=2$, $d(x)=d(y)=4$, 
     and $xh \in E(G)$;
     \item \label{t33} a $5$-cycle $xuvyt$ incident with a quadrilateral $ypzh$ such that $d(u)=d(v)=d(t)=d(z)=2$, $d(h)=3$, $d(x)=d(y)=4$, and $xh \in E(G)$;
     \item \label{t34} a $6$-cycle $xuvwyt$ with $d(u)=d(v)=d(w)=d(t)=2$ and $d(x)=d(y)=4$;
     \item \label{t35} a $6$-cycle $xuvwpq$ with $d(u)=d(v)=d(p)=d(q)=2$ and $d(x)=d(w)=4$;
     \item \label{t36} a trail $huvwpql$ with $d(h)=d(u)=d(v)=d(p)=d(q)=d(l)=2$ and $d(w)=5$,
\end{enumerate}
as illustrated in Figure \ref{fig:reducible configurations}. In this figure, as well as in all other relevant figures, we clearly indicate the neighbors of each vertex that is referenced in the proof. It is important to note that hollow vertices may overlap with one another, whereas solid vertices are always distinct from all other vertices. Furthermore, for any hollow vertex, its degree is at least $\min\{2,d\}$, where $d$ represents the number of edges incident to it as shown in the picture. 

\begin{figure}
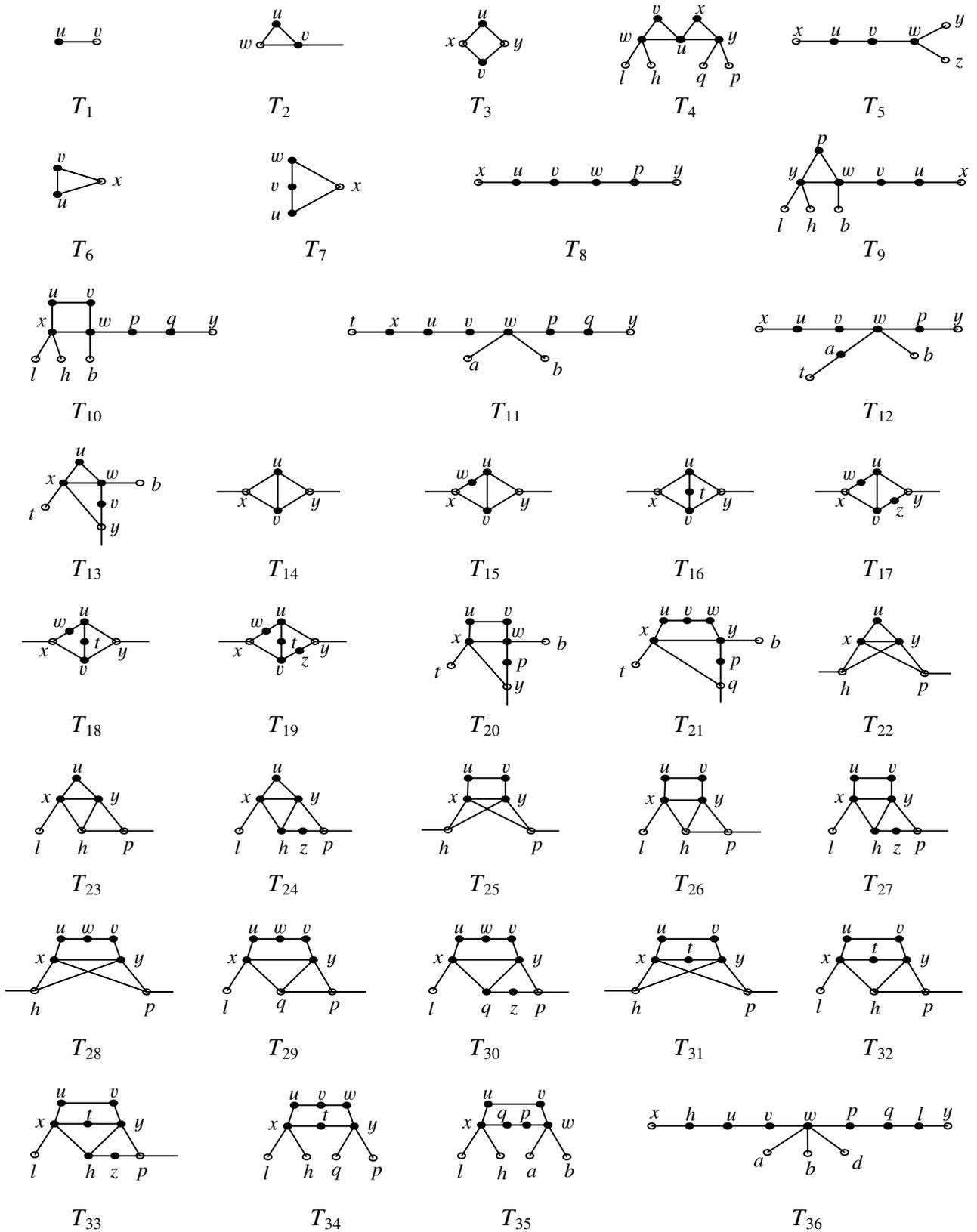

    \centering
    \include{fig-1}
    \caption{Unavailable configurations}
    \label{fig:reducible configurations}
\end{figure}

\subsection{$K_4$-minor-free graphs}

\begin{lemma} \cite{zbMATH05886363} 
\label{lem:sp-original}
Every connected $K_4$-minor-free graph $G$ with $\Delta(G) \leq 4$ contains the configuration $T_i$ for some $i \in \{1,2,3,4,6\}$ or a path $xuvy$ such that $d(u)=d(v)=2$ and $x\neq y$.
\end{lemma}

\begin{lemma} \label{lem:sp-new}
Every connected $K_4$-minor-free graph $G\not\cong C_5$ with $\Delta(G) \leq 4$ contains the configuration $T_i$ for some $i\in [12]$.
\end{lemma}

\begin{proof}
Assume that $G \not\cong C_5$ is a $K_4$-minor-free graph with $\Delta(G) \leq 4$, minimizing the number of vertices among all such graphs that violate the result.
We define the following two operations for obtaining smaller $K_4$-minor-free graphs, which can be referred to as \ref{Op1} and \ref{Op2} in subsequent discussions:
\begin{enumerate}[label=\textbf{Operation \Roman*}]\setlength{\itemsep}{-3pt}
\item \label{OOp1} If $G$ contains a $5$-cycle $xuwyv$ where $d(u)=d(v)=d(w)=2$, then we delete the vertex $v$, and if the edge between $x$ and $y$ does not already exist, connect $x$ and $y$ with an edge.
\item \label{OOp2} If \ref{OOp1} is not applied and there exists a path $xuvy$ such that $d(u)=d(v)=2$ and $x\neq y$, then replace this path with a new path  $xwy$ where $d(w)=2$.
\end{enumerate}

By Lemma \ref{lem:sp-original}, $G$ has a path $xuvy$ such that $d(u)=d(v)=2$ and $x\neq y$, so either \ref{OOp1} or \ref{OOp2} would be applied to $G$. Therefore, we get a smaller $K_4$-minor-free graph $G'$.
If $G'\cong C_5$, then $G\cong C_6$ and $G$ contains \ref{t8}. 
If $G'\not\cong C_5$, then $G'$ satisfies the lemma by the minimality of $G$.

Suppose first that $G'$ contains \ref{t3}. Since this is not the case in $G$, there would be a $5$-cycle $xuwyv$ in $G$ where $d(u)=d(v)=d(w)=2$ so that \ref{OOp2} is applied to the path $xuwy$. However, if this occurs, then \ref{OOp1} is preferentially applied and \ref{OOp2} is not activated, a contradiction. Hence, $G'$ does not contain \ref{t3}, and thus $G'$ contains $T_i$ for some $i\in [12]$ and $i\neq 3$.

We will now present just one additional case that $G'$ contains \ref{t10}, with the remaining cases being analogous. Since $G$ does not has \ref{t10}, 
either $G$ has a copy of \ref{t11} in $G$ and \ref{OOp2} is applied to the path $xuvw$, 
or $G$ has a copy of \ref{t12} with $t=x$ and \ref{OOp1} is applied to the $5$-cycle $xuvwa$.
Therefore, the configuration \ref{t10} in $G'$ is derived from $G$ either via \ref{OOp1}, provided that $G$ contains \ref{t11}, or via \ref{OOp2}, provided that $G$ contains \ref{t12}. We populate Table \ref{tab:o1p} with this result, from which analogous conclusions follow when $G'$ contains other configurations. This forms a final contradiction since $G$ now contains the configuration $T_i$ for some $i\in [12]$ and $i\neq 3$. 
\end{proof}

\begin{table} \label{tab:o1p}
    \centering
  
\begin{tabular}{|c|c|c|c|c|c|} 
    \hline
    & \textbf{Operation I} & \textbf{Operation II} & & \textbf{Operation I} & \textbf{Operation II} \\ \hline
    \ref{t1} & - & - & 
    \ref{t2} & - & \ref{t5} \\ \hline
    \ref{t4} & - & \ref{t9}  & 
    \ref{t5} & - & -  \\ \hline
    \ref{t6} & - & \ref{t7} & 
    \ref{t7} & - & \ref{t8}  \\ \hline
    \ref{t8} & - & - &  
    \ref{t9} & - & \ref{t10}  \\ \hline
    \ref{t10} & \ref{t12} & \ref{t11} &  
    \ref{t11} & - & -  \\ \hline
    \ref{t12} & - & - &  
    \ref{t13} & - & \ref{t9}, \ref{t20}  \\ \hline
    \ref{t14} & - & - &  
    \ref{t15} & - & \ref{t5}  \\ \hline
    \ref{t16} & - & \ref{t5} &  
    \ref{t17} & - & \ref{t5}  \\ \hline
    \ref{t18} & - & \ref{t5} &  
    \ref{t19} & - & \ref{t5}  \\ \hline
    \ref{t20} & \ref{t12} & \ref{t10}, \ref{t21} &  
    \ref{t21} & - & \ref{t8}, \ref{t11}  \\ \hline
    \ref{t22} & - & \ref{t25} &  
    \ref{t23} & - & \ref{t26}  \\ \hline
    \ref{t24} & - & \ref{t5}, \ref{t27} &  
    \ref{t25} & \ref{t31} & \ref{t28}  \\ \hline
    \ref{t26} & \ref{t32} & \ref{t29} &  
    \ref{t27} & \ref{t33} & \ref{t5}, \ref{t30}  \\ \hline
    \ref{t28} & - & \ref{t8} &  
    \ref{t29} & - & \ref{t8}  \\ \hline
    \ref{t30} & - & \ref{t5}, \ref{t8} &  
    \ref{t31} & - & \ref{t34}, \ref{t35}  \\ \hline
    \ref{t32} & - & \ref{t34}, \ref{t35} &  
    \ref{t33} & - & \ref{t5}, \ref{t34}, \ref{t35}  \\ \hline
    \ref{t34} & - & \ref{t8}, \ref{t11} &  
    \ref{t35} & - & \ref{t11}  \\ \hline
\end{tabular}

    \caption{Illustrations for the proofs of Lemmas \ref{lem:sp-new} and \ref{lem:o1p-new}}
\end{table}
\subsection{Outer-1-planar graphs}

\begin{lemma}\label{lem:o1p-original}
\cite[Theorems 3 and 4]{zbMATH07465224}
Every connected outer-$1$-planar graph $G\not\cong C_5$ with $|V(G)|\geq 4$  and $\Delta(G) \leq 4$ contains the configuration $T_i$ for some $i\in\{1,2,3,4,6\}\cup [13,19]\cup [22,24]$ or a path $xuvy$ such that $d(u) = d(v)=2$ and $x\neq y$. Moreover, 
\begin{itemize}
    \item if $G$ contains the latter configuration, then replacing the path $xuvy$ with a new path $xwy$ where $d(w) = 2$ still results in an outer-$1$-planar graph;
    \item ($T_3$-condition) if $G$ contains the configuration $T_3$, then replacing the path $uy$ with $uwy$ and adding the edge $xy$ (if absent) yields an outer-1-planar graph.
\end{itemize}
\end{lemma}

     

\begin{lemma} \label{lem:o1p-new}
Every connected outer-$1$-planar graph $G\not\cong C_5$ with $\Delta(G) \leq 4$ contains the configuration $T_i$ for some $i\in [35]$. Moreover, if $G$ contains the configuration $T_3$, then $T_3$-condition is satisfied.
\end{lemma}

\begin{proof}
Assume that $G \not\cong C_5$ is  an outer-1-planar graph with $|V(G)| \geq 4$ and $\Delta(G) \leq 4$, minimizing the number of vertices among all such graphs that violate the result.
We define the following two operations for obtaining smaller outer-1-planar graphs, which can be referred to as \ref{Op1} and \ref{Op2} in subsequent discussions:
\begin{enumerate}[label=\textbf{Operation \Roman*}]\setlength{\itemsep}{-3pt}
\item \label{Op1} If $G$ contains a $5$-cycle $xuwyv$ where $d(u)=d(v)=d(w)=2$, then we delete the vertex $v$, and if the edge between $x$ and $y$ does not already exist, connect $x$ and $y$ with an edge, provided that this addition does not disrupt the outer-1-planarity of the graph.
\item \label{Op2} If \ref{Op1} is not applied and there exists a path $xuvy$ such that $d(u)=d(v)=2$ and $x\neq y$, then replace this path with a new path  $xwy$ where $d(w)=2$.
\end{enumerate}

By Lemma \ref{lem:o1p-original}, $G$ has a path $xuvy$ such that $d(u)=d(v)=2$ and $x\neq y$, so either \ref{Op1} or \ref{Op2} would be applied to $G$. Therefore, we get a smaller outer-1-planar graph $G'$.
If $G'\cong C_5$, then $G\cong C_6$ and $G$ contains \ref{t8}. 
If $G'\not\cong C_5$, then $G'$ satisfies the lemma by the minimality of $G$.

 Suppose first that $G'$ contains \ref{t3} satisfying the $T_3$-condition. Since this is not the case in $G$, there would be a $5$-cycle $xuwyv$ in $G$ where $d(u)=d(v)=d(w)=2$ so that \ref{Op2} is applied to the path $xuwy$. However, by the $T_3$-condition, adding a new edge $xy$ (if absent) does not disrupt the outer-1-planarity of the graph, thus \ref{Op1} is preferentially applied and \ref{Op2} is not activated, a contradiction. Hence, $G'$ does not contain \ref{t3} satisfying the $T_3$-condition, and thus $G'$ contains $T_i$ for some $i\in [35]$ and $i\neq 3$.

We will now present just one additional case that $G'$ contains \ref{t20}, with the remaining cases being analogous. Since $G$ does not has \ref{t20}, 
either $G$ has a copy of \ref{t21} in $G$ and \ref{Op2} is applied to the path $xuvw$, or
$G$ has a copy of \ref{t10} with $xy \in E(G)$ and \ref{Op2} is applied to the path $xwpy$,
or $G$ has a copy of \ref{t12} with $t=x$ and $xy\in E(G)$ and \ref{Op1} is applied to the $5$-cycle $xuvwa$.
Therefore, the configuration \ref{t20} in $G'$ is derived from $G$ either via \ref{Op1}, provided that $G$ contains \ref{t12}, or via \ref{Op2}, provided that $G$ contains either \ref{t10} or \ref{t21}. We populate Table \ref{tab:o1p} with this result, from which analogous conclusions follow when $G'$ contains other configurations. This forms a final contradiction since $G$ now contains the configuration $T_i$ for some $i\in [35]$ and $i\neq 3$. 
\end{proof}

\subsection{Planar graphs with large girth}

\begin{lemma} \label{lem:planar12}
Every planar graph $G$ with girth at least $12$ contains the configuration $T_i$ for some $i\in \{1,5,8,11,12,36\}$.

\end{lemma}

\begin{proof}
In this proof, we define a \textit{$k$-vertex} as a vertex of degree exactly $k$, and a \textit{$k^+$-vertex} as a vertex of degree at least $k$; a \textit{$t$-thread} is a path of the form $ux_1x_2\cdots x_tv$ where each $x_i$ is a $2$-vertex while $u,v$ are $3^+$-vertices, and we call each $x_i$ a \textit{thread-2-vertex} of $u$.
Suppose, on the contrary, that $G$ contains none of the configurations. We have the following structural properties:
\begin{enumerate}[label=$P_{\arabic*}$]\setlength{\itemsep}{-3pt}
    \item \label{p1} $\delta(G)\geq 2$  (by the absence of \ref{t1});
    \item \label{p2} there are no $t$-threads for every $t\geq 4$ (by the absence of \ref{t8});
    \item \label{p3} any $3$-vertex is not incident with a $t$-thread for every $t\geq 2$ (by the absence of \ref{t5}) --- thus, every $3$-vertex has at most three thread-2-vertices;
    \item \label{p4} if a $4$-vertex is incident with a $3$-thread, then it cannot be incident with any $2$-thread (by the absence of \ref{t11});  
and if a $4$-vertex is incident with a $2$-thread, then it cannot be incident with two $1$-threads (by the absence of \ref{t12}) --- thus, by \ref{p2}, every $4$-vertex has at most four thread-2-vertices;
   \item \label{p5} any $5$-vertex is not incident with two $3$-threads (by the absence of \ref{t36}) --- thus,by \ref{p2}, every $5$-vertex has at most eleven thread-2-vertices.
\end{enumerate}

We assign to every vertex $v$ of $G$ an initial charge $ch(v):=d(v)-\frac{12}{5}$. It follows
\[ \sum_{v\in V(G)} ch(v)\leq \sum_{v\in V(G)}\left(d(v)-\frac{12}{5}\right)+\sum_{f\in F(G)}\left(\frac{1}{5}d(f)-\frac{12}{5}\right)=-\frac{24}{5}<0\]
by Euler's formula.
Next, we distribute the charge so that there is a contradiction between the final total charge and the initial total charge. The distribution is based on the following rules:
\begin{enumerate}[label=$R_{\arabic*}$]\setlength{\itemsep}{-3pt}
     \item \label{r1} Each $3$-vertex sends $\frac{1}{5}$ to each of its thread-2-vertex;
     \item \label{r2} Each $4$-vertex sends $\frac{2}{5}$ to each of its thread-2-vertex;
     \item \label{r3} Each $5^+$-vertex sends $\frac{1}{5}$ to each of its thread-2-vertex.
\end{enumerate}

Let $ch'(v)$ be the final charge of $v\in V(G)$. By \ref{p1}, we have $d(v)\geq 2$.

If $v$ is a $2$-vertex, then by \ref{p1} it must be a thread-$2$-vertex of two distinct vertices, unless $G$ is a cycle. However, if $G$ is a cycle, its length is at least $12$, and then $G$ would contain the forbidden configuration \ref{t8}, a contradiction. Thus, $v$ must always be a thread-$2$-vertex of two distinct vertices.
It follows $ch'(v)\geq 2-\frac{12}{5}+2\times \frac{1}{5}=0$ by \ref{r1}, \ref{r2}, and \ref{r3}. 
If $v$ is a 3-vertex, then $ch'(v)\geq 3-\frac{12}{5}-3\times \frac{1}{5}= 0$ by \ref{p3} and \ref{r1}.
If $v$ is a 4-vertex, then $ch'(v)\geq 4-\frac{12}{5}-4\times \frac{2}{5}= 0$ by \ref{p4} and \ref{r2}.
If $v$ is a 5-vertex, then $ch'(v)\geq 5-\frac{12}{5}-11\times \frac{1}{5}>0$ by \ref{p5} and \ref{r3}.
If $v$ is a $6^+$-vertex, then $ch'(v)\geq d(v)-\frac{12}{5}-3d(v)\cdot \frac{1}{5}\geq 0$ by \ref{p2} and \ref{r3}.

Now we conclude
\[ \sum_{v\in V(G)} ch(v)=\sum_{v\in V(G)} ch'(v)\geq 0,\]
a contradiction.
\end{proof}

\section{Reducibility}\label{sec2}

Henceforth, we use \textbf{\rm PCF} as an abbreviation for ``proper conflict-free".
Let  $\mathcal{G}_\ell$ be a class of graphs $G$ such that $\Delta(G)\leq \ell$ and $H\in \mathcal{G}_\ell$ for every $G\in \mathcal{G}_\ell$ and every $H\subseteq G$. If $G\in \mathcal{G}_\ell$ is not {\rm PCF}-{\rm (degree+$2$)}-choosable, and every $H\subseteq G$ with $H\not\cong C_5$ is {\rm PCF}-{\rm (degree+$2$)}-choosable, then we call $G$ a \textit{{\rm PCF}-{\rm (degree+$2$)}-critical graph} within $\mathcal{G}_\ell$.  We define a configuration as \textit{(degree+$2$)-reducible} within $\mathcal{G}_\ell$ if it is impossible for it to occur in a {\rm PCF}-{\rm (degree+$2$)}-critical graph $G\in \mathcal{G}_\ell$. 

Let $\phi$ be a \textup{PCF}-$L$-coloring of $G$, and let $u$ be a vertex such that $\alpha \in U_\phi(u, G)$. 
For any neighbor $w$ of $u$, if  $U_\phi(u, G-w) = \emptyset$, then every other neighbor of $u$ receives the same color under $\phi$, denoted by $\gamma_w$. 
Now, for any neighbor $w$ of $u$, we define the color $\tau_w(u)$ as follows:
\[
\tau_w(u) = 
\begin{cases} 
\gamma_w & \text{if } |U_\phi(u, G)|= 1\text{ and }\phi(w)= \alpha, \\
\alpha    & \text{if } |U_\phi(u, G)|= 1\text{ and }\phi(w)\neq \alpha, \\
\beta \in  U_\phi(u, G) \setminus \{\alpha\}   & \text{if } |U_\phi(u, G)|=2 \text{ and }\phi(w)= \alpha,\\
\text{NULL }  & \text{if } |U_\phi(u, G)|= 2\text{ and }\phi(w)\neq \alpha \text{ or } |U_\phi(u, G)|\geq 3.\\
\end{cases}
\]
We say that $\tau_w(u)$ is \textit{uniquely determined} if $\tau_w(u) \neq \text{NULL }$. 

For two ordered tuples $\{a_1, a_2, \ldots, a_t\}$ and $\{b_1, b_2, \ldots, b_t\}$, the notation
\[
\{a_1, a_2, \ldots, a_t\} \coloneqq \{b_1, b_2, \ldots, b_t\}
\]
signifies that $\{a_i := b_i\}$ for each \( i \in \{1, 2, \ldots, t\} \).

\subsection{Extension obstacles}

In this subsection, we focus on six irreducible configurations (see Figure \ref{fig:irreducible lemma}):
\begin{enumerate}[label=$H_{\arabic*}$]\setlength{\itemsep}{-3pt}
     \item \label{h1} a path $uvw$ with $d(u)=d(v)=d(w)=2$;
     \item \label{h2} a path $uvwp$ with $d(u)=d(v)=d(p)=2$ and $d(w)=4$;
     \item \label{h3} a path $uvwpq$ with $d(u)=d(v)=d(p)=d(q)=2$ and $d(w)=4$;
     \item \label{h4} a triangle $xuy$ with $d(u)=2$ and $d(x)=d(y)=4$;
     \item \label{h5} a $5$-cycle $xuwvy$ with $d(u)=d(v)=d(w)=2$ and $d(x)=d(y)=4$;
\end{enumerate}
We will demonstrate that each of these configurations is either reducible or possesses desirable properties, as established by the following lemmas.

\begin{figure}[h]
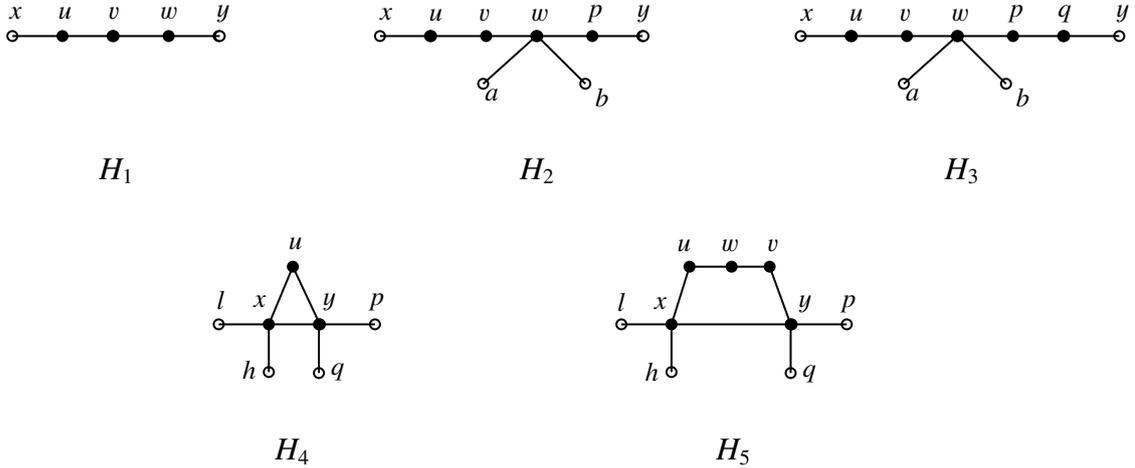

    \centering
    \include{fig-2}
    \caption{Three frequently occurring configurations}
    \label{fig:irreducible lemma}
\end{figure}


\begin{lemma}\label{lem:2-2-2}
Suppose that $G$ has \ref{h1}. Let $G'=G - \{u, v, w\}$ and let $L$ be a {\rm (degree+$2$)}-list assignment of $G$. If there is a {\rm PCF}-$L$-coloring $\phi$ of $G'$ such that 
$\{c_x,c_y\}:=\{\phi(x),\phi(y)\}$, 
then $\phi$ can be extended to a {\rm PCF}-$L$-coloring of $G$ unless $x\neq y$ and there are colors $\alpha,\beta,\xi_1,\xi_2$ such that 
\begin{itemize}
    \item  $U_\phi(x, G')=\{\alpha\}$, $U_\phi(y, G')=\{\beta\}$, 
    \item $L(u)=\{c_x, \alpha, \xi_1, \xi_2\}$,
    \item $L(v)=\{c_x, c_y, \xi_1, \xi_2\}$,
    \item $L(w)=\{c_y, \beta, \xi_1, \xi_2\}$. 
\end{itemize}
\end{lemma}


\begin{proof}
Let $\alpha \in U_\phi(x, G')$ and $\beta \in U_\phi(y, G')$.
We extend the coloring $\phi$ in the following manner.
First, we color $u$ with a color $\xi_1 \in L(u) \setminus \{c_x, \alpha\}$ and $v$ with a color $\xi_2 \in L(v) \setminus \{c_x, c_y, \xi_1\}$.
If there exists a color $\xi_3 \in L(w) \setminus \{c_y, \beta, \xi_1, \xi_2\}$, then we color $w$ with $\xi_3$, and the coloring process is successfully completed. Otherwise, we have $L(w)=\{c_y, \beta, \xi_1, \xi_2\}$. In this case, we color $w$ with $\xi_2$. 
Subsequently, if we can recolor $v$ with a color from the set $L(v)\setminus \{c_x, c_y,\xi_1, \xi_2\}$, then the coloring process is done. Otherwise, it must be that $L(v)=\{c_x, c_y,\xi_1, \xi_2\}$. At this stage, we recolor $v$ with $\xi_1$. 
Finally, if we can recolor $u$ with a color from the set $L(u)\setminus \{c_x,\alpha,\xi_1, \xi_2\}$, then the coloring is accomplished. Otherwise, it follows that $L(u)=\{c_x, \alpha, \xi_1, \xi_2\}$. Now, if there is a color $\alpha'\in U_\phi(x, G')\setminus \{\alpha\}$, then we can recolor $u$ with $\alpha$ to complete the extended coloring. Therefore, $U_\phi(x, G')=\{\alpha\}$, and by symmetry, we also have $U_\phi(y, G')=\{\beta\}$. This completes the proof.
\end{proof}

\begin{lemma}\label{lem:2-2-4-2}

Suppose that $G$ has \ref{h2}.
Let $G'=G - \{u, v, p\}$ and let $L$ be a {\rm (degree+$2$)}-list assignment of $G$. If there is a {\rm PCF}-$L$-coloring $\phi$ of $G'$ such that 
$\{c_x,c_y,c_w,c_a,c_b\}:=\{\phi(x),\phi(y),\phi(w),\phi(a),\phi(b)\}$ and $c_w\neq c_y$, then $\phi$ can be extended to a {\rm PCF}-$L$-coloring of $G$ unless $y\not\in  \{a,b\}$, and there are colors $\alpha,\beta,\xi,\zeta_1,\zeta_2$  such that
\begin{itemize}
   \item  $\{\zeta_1,\zeta_2\} = \{c_a, c_b\}$.
    \item $U_\phi(x, G')=\{\alpha\}$, $U_\phi(y, G')=\{\beta\}$,
    \item $L(u)=\{c_x, \alpha, c_w, \xi\}$,
    \item $L(v)=\{c_x, c_w, \xi, \zeta_1\}$,
    \item $L(w)=\{c_w, c_a, c_b, \tau_w(a), \tau_w(b),  c_y\}$,
    \item $L(p)=\{c_y, \beta, c_w, \zeta_2\}$.
\end{itemize}
\end{lemma}

\begin{proof}
Let $\alpha \in U_\phi(x, G')$  and $\beta \in U_\phi(y, G')$.
We extend the coloring $\phi$ in the following manner.
First, we color $u$ with a color $\xi \in L(u) \setminus \{c_x, \alpha, c_w\}$, $v$ with a color $\zeta_1 \in L(v) \setminus \{c_x, c_w, \xi\}$, and $p$ with a color $\zeta_2 \in L(p) \setminus \{c_y, \beta, c_w\}$. If $\{\zeta_1,\zeta_2\} \neq \{c_a, c_b\}$, then the coloring process is successfully completed. Otherwise, we assume $\zeta_1=c_b$ and  $\zeta_2=c_a$ (resp.\,$\zeta_1=c_a$ and  $\zeta_2=c_b$).
If we can recolor $p$ with a color in $L(p) \setminus \{c_y, \beta, c_w,c_a\}$ (resp.\,$L(p) \setminus \{c_y, \beta, c_w,c_b\}$), or $v$ with a color in $L(v) \setminus \{c_x, c_w, \xi, c_b\}$ (resp.\,$L(v) \setminus \{c_x, c_w, \xi, c_a\}$), then we are also done. Hence $L(p)=\{c_y, \beta, c_w, c_a\}$ and $L(v)=\{c_x, c_w, \xi, c_b\}$ (resp.\,$L(p)=\{c_y, \beta, c_w, c_b\}$ and $L(v)=\{c_x, c_w, \xi, c_a\}$). It follows $L(u)=\{c_x, \alpha, c_w, \xi\}$ (so $\alpha\neq c_w$). Otherwise, we can recolor $u$ with a color in  $L(u)\setminus \{c_x, \alpha, c_w, \xi\}$ and $v$ with $\xi$ to complete the desired coloring. If $U_\phi(x, G')\neq \{\alpha\}$, then recolor $u$ with $\alpha$ and $v$ with $\xi$; and if $U_\phi(y, G')\neq \{\beta\}$, then recolor $p$ with $\beta$. In each case, $\phi$ is extended successfully. Hence $U_\phi(x, G')=\{\alpha\}$ and $U_\phi(y, G')=\{\beta\}$.

Suppose first that $x\neq a,b$. 
If we can recolor $w$ with a color in $L(w) \setminus \{c_w, c_a, c_b, \tau_w(a), \tau_w(b), c_y\}$, 
then extend $\phi$ to $G$ by recoloring $p$ and $u$ with $c_w$. Hence $L(w) = \{c_w, c_a, c_b, \tau_w(a), \tau_w(b), c_y\}$ and $y\neq a,b$. 
Suppose next that $x\in \{a,b\}$.
Since $U_\phi(x, G')=\{\alpha\}$ and $\alpha\neq c_w$, 
$\tau_w(x)=\alpha$.
We recolor $u$ and $p$ with $c_w$ again. If it is possible to recolor $w$ with a color in 
$L(w) \setminus \{c_w, c_a, c_b, \tau_w(a),  \tau_w(b), c_y\}$, then we complete a {\rm PCF}-$L$-coloring of $G$.
Hence we have $L(w) = \{c_w, c_a, c_b, \tau_w(a), \tau_w(b), c_y\}$ and $y\neq a,b$ again.
\end{proof}

\begin{lemma}\label{lem:2-2-4-2-2}

Suppose that $G$ has \ref{h3}.
Let $G'=G - \{u, v, p, q\}$ and let $L$ be a {\rm (degree+$2$)}-list assignment of $G$. If there is a {\rm PCF}-$L$-coloring $\phi$ of $G'$ such that 
$\{c_x,c_y,c_w,c_a,c_b\}:=\{\phi(x),\phi(y),\phi(w),\phi(a),\phi(b)\}$,  then $\phi$ can be extended to a {\rm PCF}-$L$-coloring of $G$ unless there are colors $\alpha,\beta,\xi,\zeta_1,\zeta_2$ such that
\begin{itemize}
\item $\{\zeta_1,\zeta_2\} = \{c_a, c_b\}$,
    \item $U_\phi(x, G')=\{\alpha\}$, $U_\phi(y, G')=\{\beta\}$, 
    \item $L(u)=\{c_x, \alpha, c_w, \xi\}$, 
    \item $L(v)=\{c_x, c_w, \xi, \zeta_1\}$,
    \item $L(w)=\{c_w, c_a, c_b, \tau_w(a), \tau_w(b),  \xi\}$,
    \item $L(p)=\{c_y, c_w, \xi, \zeta_2\}$,
    \item $L(q)=\{c_y, \beta, c_w, \xi\}$.
\end{itemize}
\end{lemma}

\begin{proof}
 
Let $\alpha \in U_\phi(x, G')$  and $\beta \in U_\phi(y, G')$.
We extend the coloring $\phi$ in the following manner. 
First, we color $q$ with a color $\zeta \in L(q) \setminus \{c_w, c_y, \beta\}$ and denote this extended partial coloring by $\phi'$. By Lemma \ref{lem:2-2-4-2}, $\phi'$ can be extended successfully unless 
$U_\phi(x, G')=U_{\phi'}(x, G')=\{\alpha\}$, and there are colors $\xi,\zeta_1,\zeta_2$ such that $L(u)=\{c_x, \alpha, c_w, \xi\}$, $L(v)=\{c_x, c_w, \xi, \zeta_1\}$, $L(w)=\{c_w, c_a, c_b, \tau_w(a), \tau_w(b), \xi\}$, and $L(p)=\{c_y, \zeta, c_w, \zeta_2\}$, where  $\{\zeta_1,\zeta_2\} = \{c_a, c_b\}$.
If it is possible to recolor $q$ with a color in $L(q)\setminus \{c_w, c_y, \beta, \zeta\}$, then $\phi$ can be extended to $G$ by coloring $u,v,p$ with $\xi,\zeta_1,\zeta$, respectively. Hence 
$L(q)=\{c_w, c_y, \beta, \zeta\}$. If $ U_\phi(y, G')\neq \{\beta\}$, then recolor $q$ with $\beta$ and then accomplish the coloring by coloring $u,v,p$ with $\xi,\zeta_1,\zeta$, respectively. It follows $U_\phi(y, G')=\{\beta\}$. If $\xi\neq \zeta$, then we recolor $q$ with $c_w$, $w$ with $\xi$, and then color $u,v,p$ with $c_w,\xi_1,\zeta$, respectively. This also extends $\phi$. Hence, we have $\xi=\zeta$.
\end{proof}

\begin{lemma}\label{lem:triangle}
Suppose that $G$ has \ref{h4}.
Let $G'=G - \{u\}$ and let $L$ be a {\rm (degree+$2$)}-list assignment of $G$. 
If there is a {\rm PCF}-$L$-coloring $\phi$ of $G'$ such that 
$\{c_x,c_y,c_h,c_l,c_p,c_q\}:=\{\phi(x),\phi(y),\phi(h),\phi(l),\phi(p),\phi(q)\}$, 
then $\phi$ can be extended to a {\rm PCF}-$L$-coloring of $G$ unless there are colors $\alpha,\beta$ such that
\begin{itemize}
    \item $U_\phi(x, G')=\{\alpha\}$, $U_\phi(y, G')=\{\beta\}$, 
    \item $L(u) =\{c_x, c_y, \alpha,\beta\}$,
    \item $L(x)=\{c_x, c_y, \alpha,\beta, \tau_x(h), \tau_x(l)\}$,
    \item $L(y)=\{c_x, c_y, \alpha,\beta, \tau_y(p), \tau_y(q)\}$.
\end{itemize}
\end{lemma}

\begin{proof}
Let $\alpha \in U_\phi(x, G')$  and $\beta \in U_\phi(y, G')$.
We extend the coloring $\phi$ in the following manner. 
If $|U_\phi(x, G')|=3$ or $|U_\phi(y, G')|=3$, then $\phi$ can be extended to $G$ by coloring $u$ with a color in $L(u) \setminus \{c_x, c_y, \beta\}$ or $L(u) \setminus \{c_x, c_y, \alpha\}$. Hence we assume $U_\phi(x, G')=\{\alpha\}$ and $U_\phi(y, G')=\{\beta\}$. If $L(u)\neq \{c_x, c_y, \alpha,\beta\}$, then $\phi$ can be extended to $G$ by coloring $u$ with a color in $L(u) \setminus \{c_x, c_y, \alpha,\beta\}$. 
Hence we further assume $L(u)= \{c_x, c_y, \alpha,\beta\}$, $\alpha=c_h, c_y=c_l$, $\beta=c_p$, and $c_x=c_q$. Now, if it is possible to recolor $x$ with a color in $L(x) \setminus  \{c_x, c_y, \alpha,\beta, \tau_x(h), \tau_x(l)\}$, then we can extend $\phi$ to $G$ by coloring $u$ with $c_x$. Thus, $L(x) = \{c_x, c_y, \alpha,\beta, \tau_x(h), \tau_x(l)\}$. 
Similarly, we have $L(y)=\{c_x, c_y, \alpha,\beta,  \tau_y(p), \tau_y(q)\}$. 
\end{proof}





\begin{lemma}\label{lem:5-cycle(4-2-2-2-4)}
Suppose that $G$ has \ref{h5}. Let $G'=G - \{u, v, w\}$ and let $L$ be a {\rm (degree+$2$)}-list assignment of $G$. 
If there is a {\rm PCF}-$L$-coloring $\phi$ of $G'$ such that $\{c_x,c_y,c_h,c_l,c_p,c_q\}:=\{\phi(x),\phi(y),\phi(h),\phi(l),\phi(p),\phi(q)\}$, 
then $\phi$ can be extended to a {\rm PCF}-$L$-coloring of $G$ unless there are colors $\alpha,\beta,\zeta_1,\zeta_2,\xi_1,\xi_2$ such that 
\begin{itemize}
    \item $U_\phi(x, G')=\{\alpha\}$, $U_\phi(y, G')=\{\beta\}$,
    \item $c_h=c_l=\zeta_1$, $c_p=c_q=\zeta_2$, 
    \item $L(u) = L(v) = L(w)=\{\alpha,\beta, \xi_1, \xi_2\}$, 
    \item $L(x)=\{\alpha,\beta,\zeta_1, \tau_x(h), \tau_x(l), \xi_i\}$ for some $i=1,2$,
    \item $L(y)=\{\alpha,\beta,\zeta_2, \tau_y(p), \tau_y(q), \xi_i\}$ for some $i=1,2$.
\end{itemize}
\end{lemma}

\begin{proof}
Let $\alpha \in U_\phi(x, G')$  and $\beta \in U_\phi(y, G')$.
We extend the coloring $\phi$ in the following manner.
By Lemma \ref{lem:2-2-2}, $\phi'$ can be extended successfully unless $U_\phi(x, G')=\{\alpha\}$, $U_\phi(y, G')=\{\beta\}$, and there are two colors $\xi_1$ and $\xi_2$ such that $L(u)=\{c_x, \alpha, \xi_1, \xi_2\}$, $L(v)=\{c_x, c_y, \xi_1, \xi_2\}$, and $L(w)=\{c_y, \beta, \xi_1, \xi_2\}$. By symmetry, we consider three cases.



Case $1$: $\alpha=c_h, c_y=c_l$, $\beta=c_p$, and $c_x=c_q$. 

Let $\gamma\in L(x) \setminus \{c_x, c_y, \alpha, \tau_x(h), \tau_x(l)\}$.
If $\gamma=\xi_1$ (or $\gamma=\xi_2$), then $\phi$ can be extended to $G$ by recoloring $x$ with $\gamma$ and coloring $u,v,w$ with $\xi_2,c_x,\xi_1$ (or $\xi_1,c_x,\xi_2$), respectively. 
If $\gamma \neq \xi_1,\xi_2$, then 
$\phi$ can be extended to $G$ by recoloring $x$ with $\gamma$ and coloring $u,v,w$ with $\xi_1,c_x,\xi_2$, respectively, because $\xi_2\not\in\{c_x,\beta\}$. 

Case $2$: $\alpha=c_h, c_y=c_l$, $\beta=c_x$, and $c_p=c_q=\zeta_2$. 

Let $\theta\in L(y) \setminus \{c_x, c_y, \zeta_2, \tau_y(p), \tau_y(q)\}$. 
If $\theta=\xi_1$ (or $\theta=\xi_2$), then $\phi$ can be extended to $G$ by recoloring $y$ with $\theta$ and coloring $u,v,w$ with $\xi_1,c_y,\xi_2$ (or $\xi_2,c_y,\xi_1$), respectively.  If $\theta \neq \xi_1,\xi_2$, then $\phi$ can be extended to $G$ by recoloring $y$ with $\theta$ and coloring $u,v,w$ with $\xi_1,c_y,\xi_2$, respectively, because $\xi_1\not\in \{c_y,\alpha\}$.

Case $3$: $\alpha=c_y, c_h=c_l=\zeta_1$, $\beta=c_x$, and $ c_p=c_q=\zeta_2$.

Let $\xi\in L(x) \setminus \{\alpha,\beta,\zeta_1,\tau_x(h),\tau_x(l)\}$. 
If $\xi \neq \xi_1,\xi_2$, then $\phi$ can be extended to $G$ by recoloring $x$ with $\xi$ and coloring $u,v,w$ with $\xi_1,c_x,\xi_2$, respectively. Thus, $L(x)=\{\alpha,\beta, \zeta_1, \tau_x(h), \tau_x(l), \xi_1\}$ or $L(x)=\{\alpha,\beta, \zeta_1, \tau_x(h), \tau_x(l), \xi_2\}$. 
By symmetry, $L(y)=\{\alpha,\beta, \zeta_2, \tau_y(p), \tau_y(q), \xi_1\}$ or $L(y)=\{\alpha,\beta, \zeta_2, \tau_y(p), \tau_y(q), \xi_2\}$.
\end{proof}

\subsection{Reducing three frequently occurring configurations}

In this subsection, we begin by analyzing three configurations that frequently occur in later scenarios (see Figure \ref {fig:reducible lemma}): 
\begin{enumerate}[label=$F_{\arabic*}$]\setlength{\itemsep}{-3pt}
     \item \label{f1} two adjacent triangles $xuy$ and $xyh$ such that $d(u)=2$ and $d(x)=d(y)=4$;
     \item \label{f2} a $4$-cycle $xuvy$ adjacent to a triangle $xyh$ such that  $d(u)=d(v)=2$ and $d(x)=d(y)=4$;
     \item \label{f3} a $5$-cycle $xuvyt$ adjacent to a $4$-cycle $xtyh$ such that $d(u)=d(v)=d(t)=2$ and $d(x)=d(y)=4$,
\end{enumerate}

\begin{figure}[h]
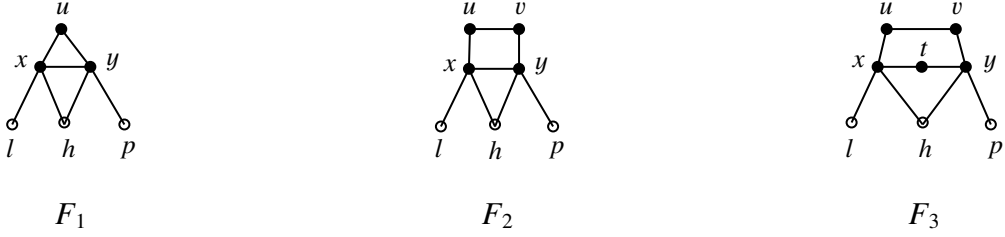

    \centering
    \include{fig-3}
    \caption{Three frequently occurring configurations}
    \label{fig:reducible lemma}
\end{figure}

\begin{lemma} \label{Reduce-F1}
   The configuration \ref{f1} is  {\rm (degree+$2$)}-reducible within $\mathcal{G}_\ell$ for every $\ell\geq 4$.
\end{lemma}

\begin{proof}
Let $G\in \mathcal{G}_\ell$ be a {\rm PCF}-{\rm (degree+$2$)}-critical graph and let $L$ be a {\rm (degree+$2$)}-list assignment of $G$. 
If $G$ contains \ref{f1}, then $G':=G-\{u\}$ has a {\rm PCF}-$L$-coloring $\phi$ as $G'\not\cong C_5$.
If $|U_\phi(x,G')|=|U_\phi(y,G')|=1$, then $U_\phi(x,G')=U_\phi(y,G')=\{\phi(h)\}$. This implies that $\phi$ can be extended to $G$ by Lemma \ref{lem:triangle}.
\end{proof}

\begin{lemma} \label{Reduce-F2}
   The configuration \ref{f2} is  {\rm (degree+$2$)}-reducible within $\mathcal{G}_\ell$ for every $\ell\geq 4$.
\end{lemma}

\begin{proof}
Let $G\in \mathcal{G}_\ell$ be a {\rm PCF}-{\rm (degree+$2$)}-critical graph and let $L$ be a {\rm (degree+$2$)}-list assignment of $G$. 
If $G$ contains \ref{f2}, then $G'=G-\{u,v\}$ has a {\rm PCF}-$L$-coloring $\phi$ as $G'\not\cong C_5$.
Let $\{c_x,c_y,c_l,c_h,c_p\}:=\{\phi(x),\phi(y),\phi(l),\phi(h),\phi(p)\}$, $\alpha\in U_\phi(x,G')$ and $\beta\in U_\phi(y,G')$.

If $|U_\phi(x, G')|=3$, then we color $v$ with $\zeta\in L(v) \setminus \{c_x, c_y,\beta\}$ and $u$ with a color in $L(u) \setminus \{c_x, c_y, \zeta\}$. This extends $\phi$. Similarly, $\phi$ can be extended if $|U_\phi(y, G')|=3$. Thus, we assume $|U_\phi(x,G')|=|U_\phi(y,G')|=1$. 
If $\alpha=c_y$, then $\phi$ can be extended to $G$ by coloring $v$ with a color $\theta$ in $L(v) \setminus \{c_x,c_y,\beta\}$ and $u$ with a color in $L(u) \setminus \{c_x,c_y,\theta\}$. 
Thus, $\alpha \neq c_y$, and by symmetry, $\beta\neq c_x$. 
It follows $\alpha=\beta=c_h$. If it is possible to color $u$ with $\xi_1\in L(u)\setminus \{c_x,c_y,\alpha\}$ and $v$ with $\xi_2\in L(v)\setminus \{c_x,c_y,\alpha\}$ such that $\xi_1\neq \xi_2$, then $\phi$ is extended successfully. Thus, we have $L(u) = L(v)= \{c_x,c_y,\alpha,\xi\}$ for some color $\xi$. In this final case, we recolor 
$x$ with a color in $L(x) \setminus \{c_x,c_y,\alpha,\tau_x(l),\tau_x(h)\}$ and then color $u,v$ with $\xi,c_x$. This also extends $\phi$.
\end{proof}




\begin{lemma} \label{Reduce-F3}
   The configuration \ref{f3} is  {\rm (degree+$2$)}-reducible within $\mathcal{G}_4$.
\end{lemma}

\begin{proof}
Let $G\in \mathcal{G}_\ell$ be a {\rm PCF}-{\rm (degree+$2$)}-critical graph and let $L$ be a {\rm (degree+$2$)}-list assignment of $G$. 
If $G$ contains \ref{f3}, then $G':=G-\{u,v\}$ has a {\rm PCF}-$L$-coloring $\phi$ as $G'\not\cong C_5$. 
Let $\{c_x,c_y,c_h,c_l,c_p,c_t\}:=\{\phi(x),\phi(y),\phi(h),\phi(l),\phi(p),\phi(t)\}$.

Let $\alpha \in U_\phi(x, G')$  and $\beta \in U_\phi(y, G')$. If we can color $u$ with a color in $L(u)\setminus \{c_x,c_y,\alpha\}$ and $v$ with a color in $L(v)\setminus \{c_x,c_y,\beta\}$ so that $u$ and $v$ receive distinct colors, then $\phi$ is extended to $G$ successfully. Thus, there is a color $\zeta_1$ such that $L(u)=\{c_x,c_y,\alpha,\zeta_1\}$ and $L(v)=\{c_x,c_y,\beta,\zeta_1\}$. If $U_\phi(x, G')\neq \{\alpha\}$, then coloring $u$ with $\alpha$ and $v$ with $\zeta_1$, we obtain a {\rm PCF}-$L$-coloring of $G$. Hence $U_\phi(x, G')= \{\alpha\}$ and $U_\phi(y, G')= \{\beta\}$ by symmetry. 

We first establish proposition $\mathcal{Q}$ that there exists no color $\theta$ satisfying $L(x)=\{c_x,c_y,c_t,\theta,\tau_x(h),\tau_x(l)\}$ and $L(y)=\{c_x,c_y,c_t,\theta,\tau_y(h),\tau_y(p)\}$ simultaneously. Assume such a $\theta$ exists; then we must have $\tau_x(h)\notin\{c_x,c_y,\mathrm{NULL}\}$ and $\tau_y(h)\notin\{c_x,c_y,\mathrm{NULL}\}$. 
If $d(h)=2$, then  $U_\phi(h,G')=\{c_x,c_y\}$, implying $\tau_x(h)=c_y$, a contradiction.
If $d(h)=3$, then let $N(h)=\{x,y,a\}$ and $\{c_a\}:=\{\phi(a)\}$. If $c_a \notin \{c_x,c_y\}$, then $|U_\phi(h,G')|=3$, implying $\tau_x(h)={\rm NULL}$; if $c_a=c_x$, then $\tau_x(h)=c_y$; and if $c_a=c_y$, then $\tau_x(h)={\rm NULL}$, a contradiction. 
If $d(h)=4$, then let $N(h)=\{x,y,a,b\}$ and $\{c_a,c_b\}:=\{\phi(a),\phi(b)\}$.
Clearly, it is impossible that $\{c_a,c_b\}= \{c_x,c_y\}$. 
Thus, we may assume $c_a=\xi$ for some color $\xi\neq c_x,c_y$.
If $c_b=\xi$, then $\tau_x(h)= c_y$; if $c_b=c_x$, then $\tau_x(h)= {\rm NULL}$; if $c_b=c_y$, then $\tau_y(h)= {\rm NULL}$; and if $c_b\neq \xi,c_x,c_y$, then $|U_\phi(h,G')|\ge 3$, implying $\tau_x(h)={\rm NULL}$. Each of these cases leads to a contradiction.

Case $1$: $\alpha=\beta$.

If $c_t=\alpha$, then $c_l=c_h=c_p\neq \alpha$. We set $c_h=\eta$. 
If $L(x)\not=\{c_x,c_y,c_t,\eta,\tau_x(h),\tau_x(l)\}$, then recolor $x$ with a color in $L(x)\setminus \{c_x,c_y,c_t,\eta,\tau_x(h),\tau_x(l)\}$, color $u,v$ with $c_x,\zeta_1$, respectively. 
This results in a {\rm PCF}-$L$-coloring of $G$. Hence $L(x)=\{c_x,c_y,c_t,\eta,\tau_x(h),\tau_x(l)\}$ and $L(y)=\{c_x,c_y,c_t,\eta,\tau_y(p),\tau_y(q)\}$ by symmetry. However, this contradicts the proposition $\mathcal{Q}$. 

If $c_h = \alpha$, then $c_l = c_p = c_t\neq \alpha$; and if $c_l = c_p = \alpha$, then $c_t = c_h\neq \alpha$. In either case,
whenever $L(x)\not=\{c_x,c_y,c_t,\alpha,\tau_x(h),\tau_x(l)\}$, we recolor $x$ with a color in $L(x)\setminus \{c_x,c_y,c_t,\alpha,\tau_x(h),\tau_x(l)\}$, and then color $u,v$ with $c_x,\zeta_1$, respectively. 
This results in a {\rm PCF}-$L$-coloring of $G$. Hence $L(x)=\{c_x,c_y,c_t,\alpha,\tau_x(h),\tau_x(l)\}$ and $L(y)=\{c_x,c_y,c_t,\alpha,\tau_y(p),\tau_y(q)\}$ by symmetry. However, this contradicts the proposition $\mathcal{Q}$.

Case $2$: $\alpha\not=\beta$.

If $c_t=\alpha$, then $c_l=c_h=\beta$ and $c_p=\alpha$. 
We recolor $t$ with a color in $L(t)\setminus \{c_x,c_y,\alpha\}$. To complete a {\rm PCF}-$L$-coloring of $G$, we just need to color $u$ with $\alpha$ and $v$ with $\zeta_1$, respectively. Thus, $c_t \neq \alpha$, and by symmetry, $c_t \neq \beta$. It follows $c_l=\alpha$, $c_p=\beta$, and $c_t=c_h$. In this case, $\phi$ can be extended to $G$ by recoloring $t$ with a color in $L(t) \setminus \{c_x,c_y,c_t\}$, and coloring $u,v$ with $\alpha,\beta$, respectively.
\end{proof}

\subsection{Reducing each $\boldsymbol{T}_i$} \label{subsection:Ti}

\begin{lemma}
    The configuration \ref{t1} is {\rm (degree+$2$)}-reducible within $\mathcal{G}_\ell$ for every $\ell\geq 4$.
\end{lemma}

\begin{proof}
Let $G\in \mathcal{G}_\ell$ be a {\rm PCF}-{\rm (degree+$2$)}-critical graph and let $L$ be a {\rm (degree+$2$)}-list assignment of $G$. Assume $uv\in E(G)$. If $G$ contains \ref{t1}, then let $G':=G-\{u\}$. If $G'\cong C_5$, then $G\not\cong C_5$ and $\Delta(G)\leq 3$. It follows that $G$ is {\rm PCF}-{\rm (degree+$2$)}-choosable by Theorem \ref{thm:cubic}\ref{two}, a contradiction. Thus, $G'\not\cong C_5$ and it has a {\rm PCF}-$L$-coloring $\phi$. Let $\{c_v\}:=\{\phi(v)\}$. This coloring can be extended to $G$ by coloring $u$ with a color in
$L(u) \setminus \{\phi(v), \alpha\}$, where $\alpha$ is an arbitrary color in $U_\phi(v,G')$. 
\end{proof}

\begin{lemma}
   The configuration \ref{t2} is {\rm (degree+$2$)}-reducible within $\mathcal{G}_\ell$ for every $\ell\geq 4$.
\end{lemma}

\begin{proof}
Let $G\in \mathcal{G}_\ell$ be a {\rm PCF}-{\rm (degree+$2$)}-critical graph and let $L$ be a {\rm (degree+$2$)}-list assignment of $G$. 
If $G$ contains \ref{t2}, then let $G':=G-\{u\}$.
If $G'\cong C_5$, then $G\not\cong C_5$ and $\Delta(G)\leq 3$. It follows that $G$ is {\rm PCF}-{\rm (degree+$2$)}-choosable by Theorem \ref{thm:cubic}\ref{two}, a contradiction. Thus, $G'\not\cong C_5$ and it has a {\rm PCF}-$L$-coloring $\phi$. 
Let $\{c_v,c_w\}:=\{\phi(v),\phi(w)\}$.  
Choose $\alpha$ be an arbitrary color in $U_\phi(w,G')$ (note that it is possible that $\alpha=c_v$). We extend $\phi$ to a {\rm PCF}-$L$-coloring of $G$ by coloring $u$ with a color $\gamma \in L(u) \setminus \{c_v, c_w,\alpha\}$. One can easily see that either $\gamma$ or $c_w$ is a PCF-color of $v$ under the extended coloring of $G$ because $d(v)=3$.
\end{proof}

\begin{lemma}
   The configuration \ref{t3} is  {\rm (degree+$2$)}-reducible within $\mathcal{G}_4$.
\end{lemma}

\begin{proof}
Let $G\in \mathcal{G}_4$ be a {\rm PCF}-{\rm (degree+$2$)}-critical graph and let $L$ be a {\rm (degree+$2$)}-list assignment of $G$. 
If $G$ contains \ref{t3}, then let $G':=G-\{v\}$. 
If $G'\cong C_5$, then $G\not\cong C_5$ and $\Delta(G)\leq 3$. It follows that $G$ is {\rm PCF}-{\rm (degree+$2$)}-choosable by Theorem \ref{thm:cubic}\ref{two}, 
a contradiction. Thus, $G'\not\cong C_5$ and it has a {\rm PCF}-$L$-coloring $\phi$. 
Assume $d(x)\leq d(y)$. We consider two distinct cases.

Case 1. $d(x)=3$.

Assume $N(x)=\{u,v,p\}$. 
Let $\{c_x,c_u,c_y,c_p\}:=\{\phi(x),\phi(u),\phi(y),\phi(p)\}$. Since $x$ has degree 2 in $G'$, $c_u\neq c_p$.
Choose $\beta$ be an arbitrary color in $U_\phi(y,G')$ (note that it is possible that $\beta=c_u$). We can extend $\phi$ to a {\rm PCF}-$L$-coloring of $G$ by coloring $v$ with a color in $L(v) \setminus \{c_x, c_y,\beta\}$. One can easily see that either $c_u$ or $c_p$ is a PCF-color of $x$ under the extended coloring of $G$ as $d(x)=3$.

Case 2. $d(x)=4$.

Assume $N(x)=\{u,v,p,q\}$. 
Let $\{c_x,c_u,c_y,c_p,c_q\}:=\{\phi(x),\phi(u),\phi(y),\phi(p),\phi(q)\}$. 
Choose $\alpha$ be an arbitrary color in $U_\phi(x,G')$ and $\beta$ be an arbitrary color in $U_\phi(y,G')$. 
If $L(v)\neq \{c_x,c_y,\alpha,\beta\}$, then we extend $\phi$ to a {\rm PCF}-$L$-coloring of $G$ by coloring $v$ with a color in $L(v) \setminus \{c_x, c_y,\alpha,\beta\}$. 
Thus, $L(v) = \{c_x, c_y,\alpha,\beta\}$. Since we now have $\alpha\neq \beta$, we may assume $\beta \neq c_u$. This implies that $y$ is adjacent to two vertices colored $c_u$ and one vertex colored $\beta$ in $G'$.

Suppose first that $\alpha=c_u$. It follows $c_p=c_q$. We color $v$ with $\alpha$ first. If it is possible to recolor $u$ with a color in $L(u) \setminus \{c_x, c_y,c_u,\beta\}$, then $\phi$ is  extended successfully because now $\alpha$ and $\beta$ are still PCF-colors of $x$ and $y$, respectively. Thus, $L(u)= \{c_x, c_y,c_u,\beta\}$.
We recolor $u$ with $c_x$.
If it is possible to recolor $x$ with a color in $L(x)\setminus \{c_x,c_y,\alpha,c_p,\tau_x(p),\tau_x(q)\}$, then we complete the extension of $\phi$. 
Thus, $L(x)= \{c_x,c_y,\alpha,c_p,\tau_x(p),\tau_x(q)\}$. At this stage, we recolor $x$ with $\alpha$ and $v$ with $\beta$. This completes a {\rm PCF}-$L$-coloring $\phi$ of $G$ such that $U_\phi(x,G)=\{c_x,\beta\}$ and $U_\phi(y,G)=\{c_x,\alpha\}$.

Suppose next that $\alpha \neq c_u$. It follows that  $x$ is adjacent to two vertices colored $c_u$ and one vertex colored $\alpha$ in $G'$. We color $v$ with $\alpha$ first.
Now, we can extend $\phi$ to a {\rm PCF}-$L$-coloring of $G$ such that $c_u\in U_\phi(x,G) \cap U_\phi(y,G)$ by recoloring $u$ with a color in $L(u) \setminus \{c_x, c_y, c_u\}$.
\end{proof}

\begin{lemma}
    The configuration \ref{t4} is {\rm (degree+$2$)}-reducible within $\mathcal{G}_4$.
\end{lemma}

\begin{proof}
Let $G\in \mathcal{G}_\ell$ be a {\rm PCF}-{\rm (degree+$2$)}-critical graph and let $L$ be a (degree+2)-list assignment of $G$. 
If $G$ contains \ref{t4}, then $G':=G-\{v,x\}$ has a {\rm PCF}-$L$-coloring $\phi$ as $G' \not\cong C_5$.
Let $\{c_u,c_{w},c_{y}\}:=\{\phi(u),\phi(w),\phi(y)\}$.
Choose $\alpha\in U_\phi(w,G')$ and $\beta\in U_\phi(y,G')$. 

We color $v$ with a color $\gamma \in L(v) \setminus \{c_{w}, c_u, \alpha\}$ and $x$ with a color $\zeta \in L(x) \setminus \{c_{y}, c_u, \beta\}$.
This extends $\phi$ to $G$ unless $\{\gamma, \zeta\} = \{c_{w}, c_{y}\}$. In this exceptional case, we have $\gamma=c_{y}$ and $\zeta=c_{w}$.
It follows that $L(v)=\{c_{w}, c_u, \alpha, c_{y}\}$ and $L(x)=\{c_{y}, c_u, \beta, c_{w}\}$ because otherwise we can recolor $v$ or $x$ to obtain a {\rm PCF}-$L$-coloring of $G$ again.
Thus, $c_u\neq \alpha,\beta$ and we can complete a {\rm PCF}-$L$-coloring of $G$ by recoloring $u$ with a color in $L(u)\setminus \{c_u,c_w,c_y,\alpha,\beta\}$, and coloring both $v$ and $x$ with $c_u$. 
\end{proof}

\begin{lemma}
   The configuration \ref{t5} is  {\rm (degree+$2$)}-reducible within $\mathcal{G}_\ell$ for every $\ell\geq 4$.
\end{lemma}

\begin{proof}
Let $G\in \mathcal{G}_\ell$ be a {\rm PCF}-{\rm (degree+$2$)}-critical graph and let $L$ be a {\rm (degree+$2$)}-list assignment of $G$. 
If $G$ contains \ref{t5}, then let $G':=G-\{u, v\}$. 
If $G'\cong C_5$, then $G\not\cong C_5$ and $\Delta(G)\leq 3$. It follows that $G$ is {\rm PCF}-{\rm (degree+$2$)}-choosable by Theorem \ref{thm:cubic}\ref{two}, a contradiction. Thus, $G'\not\cong C_5$ and it has a {\rm PCF}-$L$-coloring $\phi$. 
Let $\{c_x,c_w,c_y,c_z\}:=\{\phi(x),\phi(w),\phi(y),\phi(z)\}$. 

If $x=w$, then we may assume $z=u$. Now, $\phi$ can be extended to $G$ by coloring $u$ with a color $\gamma$ in $L(u) \setminus \{c_w, c_y\}$ and $v$ with a color in $L(v) \setminus \{c_w, c_y,\gamma\}$. 
Thus, $x\neq w$ and $u\neq y,z$. It follows $c_y \neq c_z$ since $w$ has degree 2 in $G'$. Let $\alpha\in U_\phi(x,G')$. Then, $\phi$ can be extended to $G$ by coloring $u$ with a color $\zeta$ in $L(u) \setminus \{c_x, \alpha, c_w\}$ and $v$ with a color in $L(v) \setminus \{c_x,c_w,\zeta\}$. 
\end{proof}

\begin{lemma}
   The configuration \ref{t6} is  {\rm (degree+$2$)}-reducible within $\mathcal{G}_\ell$ for every $\ell\geq 4$.
\end{lemma}

\begin{proof}
Let $G\in \mathcal{G}_\ell$ be a {\rm PCF}-{\rm (degree+$2$)}-critical graph and let $L$ be a {\rm (degree+$2$)}-list assignment of $G$. 
If $G$ contains \ref{t6}, then let $G':=G-\{u\}$ has a {\rm PCF}-$L$-coloring $\phi$ as $G' \not\cong C_5$. 
Let $\{c_x,c_v\}:=\{\phi(x),\phi(v)\}$. 
Choose $\alpha$ be an arbitrary color in $U_\phi(x,G')$. 
We can extend $\phi$ to a {\rm PCF}-$L$-coloring of $G$ by coloring $u$ with a color in $L(u) \setminus \{c_x, \alpha, c_v\}$. 
\end{proof}

\begin{lemma}
   The configuration \ref{t7} is  {\rm (degree+$2$)}-reducible within $\mathcal{G}_\ell$ for every $\ell\geq 4$.
\end{lemma}

\begin{proof}
If $G$ contains \ref{t7}, then let $G':= G-\{u,v,w\}$.
If $G'\cong C_5$, then $G\not\cong C_5$ and $G$ is a connected outerplanar graph. It follows that $G$ is {\rm PCF}-{\rm (degree+$2$)}-choosable by Theorem \ref{thm-OP}, a contradiction. Thus, $G'\not\cong C_5$ and it has a {\rm PCF}-$L$-coloring $\phi$. This is an immediate corollary of Lemma \ref{lem:2-2-2} as $x\in N(u)\cap N(w)$.
\end{proof}


\begin{lemma}
   The configuration \ref{t8} is  {\rm (degree+$2$)}-reducible within $\mathcal{G}_\ell$ for every $\ell\geq 4$.
\end{lemma}

\begin{proof}
Let $G\in\mathcal{G}_\ell$ be a {\rm PCF}-{\rm (degree+$2$)}-critical graph and let $L$ be a {\rm (degree+$2$)}-list assignment of $G$. 
If $G$ contains \ref{t8}, then let $G':=G-\{u, v, w, p\}$. 
Note that $x=y$ is allowed but when this occurs we have $d(x) \neq 2$ by the definition of \ref{t8}.


If $G'\cong C_5$, then $G\not\cong C_5$ and $\Delta(G)\leq 3$. It follows that $G$ is {\rm PCF}-{\rm (degree+$2$)}-choosable by Theorem \ref{thm:cubic}\ref{two}, a contradiction. Thus, $G'\not\cong C_5$ and it has a {\rm PCF}-$L$-coloring $\phi$.
Let $\{c_x,c_y\}:=\{\phi(x),\phi(y)\}$. Since $x$ and $y$ are not isolated vertices in $G'$, we can choose $\alpha \in U_\phi(x,G')$ and $\beta \in U_\phi(y,G')$.
We first color $p$ with a color $\gamma \in L(p)\setminus \{ c_y,\beta\}$ to extend $\phi$ to $G-\{u,v,w\}$. Now applying Lemma \ref{lem:2-2-2}, if $\phi$ cannot be extended to $G$, then there are two colors $\xi_1$ and $\xi_2$ such that 
$L(u)=\{c_x, \alpha, \xi_1, \xi_2\}$,
$L(v)=\{c_x, \gamma, \xi_1, \xi_2\}$, and $L(w)=\{\gamma, c_y, \xi_1, \xi_2\}$. 
In the latter case, we can also finish a {\rm PCF}-$L$-coloring of $G$ by recoloring $p$ with $\zeta\in L(p)\setminus \{\gamma,c_y,\beta\}$ (assuming $\zeta\neq \xi_1$ by symmetry) and then coloring $u,v,w$ with $\xi_2,\xi_1,\gamma$, respectively. 
\end{proof}





\begin{lemma}
   The configuration \ref{t9} is  {\rm (degree+$2$)}-reducible within $\mathcal{G}_4$.
\end{lemma}

\begin{proof}
If $G$ contains \ref{t9}, then $G':=G-\{u,v,p\}$ has a {\rm PCF}-$L$-coloring $\phi$ as $G' \not\cong C_5$, this is an immediate corollary of Lemma \ref{lem:2-2-4-2} as $y\in N(w)\cap N(p)$. 
\end{proof}

\begin{lemma}
   The configuration \ref{t10} is  {\rm (degree+$2$)}-reducible within $\mathcal{G}_4$.
\end{lemma}

\begin{proof}
Let $G\in \mathcal{G}_4$ be a {\rm PCF}-{\rm (degree+$2$)}-critical graph and let $L$ be a {\rm (degree+$2$)}-list assignment of $G$. 
If $G$ contains \ref{t10}, then $G':=G-\{u, v, p, q\}$ has a {\rm PCF}-$L$-coloring $\phi$ as $G' \not\cong C_5$. 
Let $\{c_x,c_w,c_b,c_y\}:=\{\phi(x),\phi(w),\phi(b),\phi(y)\}$.

By Lemma \ref{lem:2-2-4-2-2}, $\phi$ can be extended to $G$ unless $\{c_h,c_l\}=\{\alpha,c_w\}$, and there is a color $\xi$ such that $L(u)=\{c_x, \alpha, c_w, \xi\}$, $L(v)=\{c_x, c_w, \xi, \zeta_1\}$, $L(w)=\{c_w, c_x, c_b, \alpha, \tau_w(b),  \xi\}$, $L(p)=\{c_y, c_w, \xi, \zeta_2\}$, and $L(q)=\{c_y, \beta, c_w, \xi\}$, where  $\{\zeta_1,\zeta_2\} = \{c_x, c_b\}$ and $U_\phi(y, G')=\{\beta\}$.
However, this is not an exceptional case currently because we can complete a {\rm PCF}-$L$-coloring of $G$ by recoloring $w$ with $\alpha$ and then coloring $u,v,p,q$ with $\xi,c_w,c_w,\xi$, respectively.
\end{proof}

\begin{lemma}
   The configuration \ref{t11} is  {\rm (degree+$2$)}-reducible within $\mathcal{G}_\ell$ for every $\ell\geq 4$.
\end{lemma}

\begin{proof}
Let $G\in \mathcal{G}_4$ be a {\rm PCF}-{\rm (degree+$2$)}-critical graph and let $L$ be a {\rm (degree+$2$)}-list assignment of $G$.
If $G$ contains \ref{t11}, then $G':=G-\{ u, v, p, q\}$ has a  {\rm PCF}-$L$-coloring $\phi$ as $G' \not\cong C_5$. 
Let $\{c_x,c_t,c_w,c_a,c_b,c_y\}:=\{\phi(x),\phi(t),\phi(w),\phi(a),\phi(b),\phi(y)\}$.
Choose $\alpha\in U_\phi(t, G')$. 
By Lemma \ref{lem:2-2-4-2-2}, this $\phi$ can be extended to $G$ unless
there is a color $\xi$ such that 
$L(u)=\{c_x, c_t, c_w, \xi\}$, 
$L(v)=\{c_x, c_w, \xi, \zeta_1\}$,
$L(w)=\{c_w, c_a, c_b, \tau_w(a), \tau_w(b),  \xi\}$,
$L(p)=\{c_y, c_w, \xi, \zeta_2\}$, and
$L(q)=\{c_y, \beta, c_w, \xi\}$, where  
$\{\zeta_1,\zeta_2\} = \{c_a, c_b\}$ and $U_\phi(y, G')=\{\beta\}$.

In this case, we color $u,v,p,q$ with  $c_x,\xi,\zeta_2,\xi$, respectively.
If we can recolor $x$ with a color in $L(x)\setminus \{c_x,c_t,\xi,\alpha\}$, then we are done.
Otherwise, we have $L(x)=\{c_x,c_t,\xi,\alpha\}$. Now, we complete a {\rm PCF}-$L$-coloring of $G$ by recoloring $x$ and $w$ with $\xi$, and $v$ and $q$ with $c_w$, respectively.
\end{proof}



\begin{lemma}
    The configuration \ref{t12} is {\rm (degree+$2$)}-reducible within $\mathcal{G}_\ell$ for every $\ell\geq 4$.
\end{lemma}

\begin{proof}
Let $G\in \mathcal{G}_\ell$ be a {\rm PCF}-{\rm (degree+$2$)}-critical graph and let $L$ be a {\rm (degree+$2$)}-list assignment of $G$.
If $G$ contains \ref{t12}, then $G':=G-\{a,u,v,p\}$ has a {\rm PCF}-$L$-coloring $\phi$ as $G' \not\cong C_5$. 
Let $\{c_x,c_t,c_w,c_b,c_y\}:=\{\phi(x),\phi(t),\phi(w),\phi(b),\phi(y)\}$.
Let $\alpha\in U_\phi(x, G')$, $\beta\in U_\phi(y, G')$, and $\gamma\in U_\phi(t, G')$. 
We first extend $\phi$ to $G'':=G-\{u, v, p\}$ by coloring $a$ with $\theta\in L(a)\setminus \{c_t,c_w,\gamma\}$.
We can assume that $c_w\neq c_y$ and $c_w \neq c_t$, otherwise we can recolor $w$ with a color in $L(w) \setminus \{\theta, c_b, c_t, \tau_w(b), c_y\}$ so that the resulting coloring satisfies this assumption.

By Lemma \ref{lem:2-2-4-2}, the coloring $\phi$ of $G''$ can be extended to $G$ unless there is one color $\xi$ such that $y \notin \{a,b\}$, $U_\phi(x,G')=\{\alpha\}$, $U_\phi(y,G')=\{\beta\}$, $L(u)=\{c_x, \alpha, c_w, \xi\}$, $L(v)=\{c_x, c_w, \xi, \zeta_1\}$, $L(w)=\{c_w, \theta, c_b, c_t, \tau_w(b), c_y\}$, and $L(p)=\{c_y, \beta, c_w, \zeta_2\}$, where  $\{\zeta_1,\zeta_2\} = \{\theta, c_b\}$. 
In this case, we color $u,v,p$ with  $\xi,c_w,c_w$, respectively and recolor $w$ with $\theta$. 
If we can recolor $a$ with a color in $L(a)\setminus \{\theta,c_b,c_t,\gamma\}$, then we complete the extension. Otherwise, we have $L(a)=\{\theta,c_b,c_t,\gamma\}$. Now, we finish a {\rm PCF}-$L$-coloring of $G$ by recoloring $a$ with $c_b$ and $v$ with $\zeta_1$. 
\end{proof}

\begin{lemma}
    The configuration \ref{t13} is {\rm (degree+$2$)}-reducible within $\mathcal{G}_4$.
\end{lemma}

\begin{proof}
Let $G\in \mathcal{G}_\ell$ be a {\rm PCF}-{\rm (degree+$2$)}-critical graph and let $L$ be a {\rm (degree+$2$)}-list assignment of $G$. 
If $G$ contains \ref{t13}, then $G':=G-\{u\}$ has a  {\rm PCF}-$L$-coloring $\phi$ as $G' \not\cong C_5$. 
Let $\{c_x,c_y,c_w,c_t,c_v,c_b\}:=\{\phi(x),\phi(y),\phi(w),\phi(t),\phi(v),\phi(b)\}$. By Lemma \ref{lem:triangle}, if $\phi$ cannot be extended to $G$, then $U_\phi(x, G')=\{\alpha\}$, $U_\phi(w, G')=\{\beta\}$, and $L(w)=\{c_w, c_x, \alpha, \beta, \tau_w(b),\tau_w(v)\}$.
Since $c_w\neq c_y$, $U_\phi(x, G')=\{\alpha\}$ implies $\alpha\in \{c_w,c_y\}$. 
Since $\tau_w(v)=c_y$, $\alpha\in \{c_w,\tau_w(v)\}$. It follows $|L(w)|\leq 5$, a contradiction.
\end{proof}

\begin{lemma}
    The configuration \ref{t14} is {\rm (degree+$2$)}-reducible within $\mathcal{G}_4$.
\end{lemma}

\begin{proof}
Let $G\in \mathcal{G}_4$ be a {\rm PCF}-{\rm (degree+$2$)}-critical graph and let $L$ be a {\rm (degree+$2$)}-list assignment of $G$. 
If $G$ contains \ref{t14}, then let $G':=G-\{v\}$.
If $G'\cong C_5$, then $G\not\cong C_5$ and $\Delta(G)\leq 3$. It follows that $G$ is {\rm PCF}-{\rm (degree+$2$)}-choosable by Theorem \ref{thm:cubic}\ref{two}, a contradiction. Thus, $G'\not\cong C_5$ and it has a {\rm PCF}-$L$-coloring $\phi$.
Let $\{c_x,c_y,c_u\}:=\{\phi(x),\phi(y),\phi(u)\}$.
Let $\alpha\in U_\phi(x,G')$ and $\beta\in U_\phi(y,G')$. 

If $L(v) \neq \{c_x,c_y,c_u,\alpha,\beta\}$, then $\phi$ can be extended to $G$ by coloring $v$ with a color in
$L(v) \setminus \{c_x,c_y,c_u,\alpha,\beta\}$. 
Thus, $L(v)=\{c_x,c_y,c_u,\alpha,\beta\}$.
We extend $\phi$ to $G$ by coloring $v$ with $\alpha$ and recoloring $u$ with $\gamma\in L(u) \setminus \{c_x,c_y,c_u,\alpha\}$.
If $N(x)=\{u,v,p\}$ or $N(x)=\{u,v,p,q\}$, then $\alpha\in U_\phi(x,G')$ and $\alpha\neq c_u$ imply $\phi(p)=\alpha$ or 
$\{\phi(p),\phi(q)\}=\{c_u,\alpha\}$, respectively. It follows that $\gamma$ is a {\rm PCF}-color of $x$. 
On the other hand, if $N(y)=\{u,v,h\}$ or $N(y)=\{u,v,h,l\}$, then $\beta\in U_\phi(y,G')$ and $\beta\neq c_u$ imply $\phi(h)=\beta$ or $\{\phi(h),\phi(l)\}=\{c_u,\beta\}$, respectively. It follows that $\alpha$ is a {\rm PCF}-color of $y$.
Finally, $\alpha\in U_\phi(u,G)$ and $\gamma\in U_\phi(v,G)$ since $\{\alpha,\gamma\}\cap \{c_x,c_y\}=\emptyset$.
\end{proof}

\begin{lemma}
    The configuration \ref{t15} is {\rm (degree+$2$)}-reducible within $\mathcal{G}_\ell$ for every $\ell \geq 4$.
\end{lemma}

\begin{proof}
Let $G\in \mathcal{G}_4$ be a {\rm PCF}-{\rm (degree+$2$)}-critical graph and let $L$ be a {\rm (degree+$2$)}-list assignment of $G$. 
If $G$ contains \ref{t15}, then let $G':=G-\{u,w\}$. 
If $G'\cong C_5$, then $G\not\cong C_5$ and $\Delta(G)\leq 3$. It follows that $G$ is {\rm PCF}-{\rm (degree+$2$)}-choosable by Theorem \ref{thm:cubic}\ref{two}, a contradiction. Thus, $G'\not\cong C_5$ and it has a {\rm PCF}-$L$-coloring $\phi$.
Let $\{c_x,c_y,c_v\}:=\{\phi(x),\phi(y),\phi(v)\}$.
Let $\alpha\in U_\phi(x,G')$ and $\beta\in U_\phi(y,G')$. 
We now extend $\phi$ to $G$ by coloring $u$ with $\gamma \in L(u) \setminus \{c_x,c_y,c_v,\beta\}$ and $w$ with $\xi\in L(w)\setminus \{c_x,\alpha,\gamma\}$. It is easy to see that each vertex besides $u$ has a {\rm PCF}-color in this coloring of $G$.  
If $u$ does not have a {\rm PCF}-color now, then $\xi=c_y=c_v$. However, this is impossible as $yv\in E(G)$.
\end{proof}


\begin{lemma}
    The configuration \ref{t16} is {\rm (degree+$2$)}-reducible within $\mathcal{G}_4$.
\end{lemma}

\begin{proof}
Let $G\in \mathcal{G}_\ell$ be a {\rm PCF}-{\rm (degree+$2$)}-critical graph and let $L$ be a {\rm (degree+$2$)}-list assignment of $G$. 
If $G$ contains \ref{t16}, then let $G':=G-\{u,t\}$.
If $G'\cong C_5$, then $G\not\cong C_5$ and $\Delta(G)\leq 3$. It follows that $G$ is {\rm PCF}-{\rm (degree+$2$)}-choosable by Theorem \ref{thm:cubic}\ref{two}, a contradiction. Thus, $G'\not\cong C_5$ and it has a {\rm PCF}-$L$-coloring $\phi$.
Let $\{c_x,c_y,c_v\}:=\{\phi(x),\phi(y),\phi(v)\}$.
Let $\alpha \in U_\phi(x,G')$ and $\beta \in U_\phi(y,G')$. 

If $L(u) \neq \{c_x,c_y,c_v,\alpha,\beta\}$, then $\phi$ can be extended to $G$ by coloring $u$ with $\gamma\in L(u) \setminus \{c_x,c_y,c_v,\alpha,\beta\}$ and $t$ with a color in $L(t) \setminus \{\gamma,c_v\}$. 
Thus, $L(u)=\{c_x,c_y,c_v,\alpha,\beta\}$.
We extend $\phi$ to $G$ by coloring $u$ with $\alpha$, recoloring $v$ with $\xi \in L(v) \setminus \{c_x,c_y,c_v,\alpha\}$ and coloring $t$ with a color in $L(t) \setminus \{\alpha,\xi\}$. If $N(x)=\{u,v,p\}$ or $N(x)=\{u,v,p,q\}$, then $\alpha \in U_\phi(x,G')$ and $\alpha \neq c_v$ implying $\phi(p)=\alpha$ or $\{\phi(p),\phi(q)\}=\{c_v,\alpha\}$, respectively. It follows that $\xi$ is a {\rm PCF}-color of $x$. On the other hand, if $N(y)=\{u,v,h\}$ or $N(y)=\{u,v,h,l\}$, then $\beta \in U_\phi(y,G')$ and $\beta \neq c_v$ implying $\phi(p)=\beta$ or $\{\phi(p),\phi(q)\}=\{c_v,\beta\}$, respectively. It follows that $\alpha$ is a {\rm PCF}-color of $y$. Finally, $\{c_x,c_y\} \cap U_\phi(u,G)$ and $\{c_x,c_y\} \cap U_\phi(v,G)$ are nonempty as $c_x \neq c_y$.  
\end{proof}

\begin{lemma}
    The configuration \ref{t17} is {\rm (degree+$2$)}-reducible within $\mathcal{G}_\ell$ for every $\ell \geq 4$.
\end{lemma}

\begin{proof}
Let $G\in \mathcal{G}_\ell$ be a {\rm PCF}-{\rm (degree+$2$)}-critical graph and let $L$ be a {\rm (degree+$2$)}-list assignment of $G$. 
If $G$ contains \ref{t17}, then let $G':=G-\{u,w\}$.
If $G'\cong C_5$, then $G\not\cong C_5$ and $\Delta(G)\leq 3$. It follows that $G$ is {\rm PCF}-{\rm (degree+$2$)}-choosable by Theorem \ref{thm:cubic}\ref{two}, a contradiction. Thus, $G'\not\cong C_5$ and it has a {\rm PCF}-$L$-coloring $\phi$.
Let $\{c_x,c_y,c_v,c_z\}:=\{\phi(x),\phi(y),\phi(v),\phi(z)\}$.
Let $\alpha \in U_\phi(x,G')$ and $\beta \in U_\phi(y,G')$. 

We now extend $\phi$ to $G$ by coloring $u$ with a color $\gamma$ in $L(u) \setminus \{c_x,c_y,c_v,\beta\}$ and $w$ with a color $\xi$ in $L(w) \setminus \{c_x,\gamma,\alpha\}$. It easy to see that each vertex besides $u,v$ has a {\rm PCF}-color in this coloring of $G$. If $u$ or $v$ does not have a {\rm PCF}-color now, then $\xi=c_v=c_y$ or $\gamma=c_x=c_z$, respectively. However, either is impossible as $c_v\neq c_y$ and $c_x\neq c_u$ by the fact that $z$ has degree two in $G'$ and $\gamma\neq c_x$ by the choice of $\gamma$.
\end{proof}

\begin{lemma}
    The configuration \ref{t18} is {\rm (degree+$2$)}-reducible within within $\mathcal{G}_\ell$ for every $\ell \geq 4$.
\end{lemma}

\begin{proof}
Let $G\in \mathcal{G}_\ell$ be a {\rm PCF}-{\rm (degree+$2$)}-critical graph and let $L$ be a {\rm (degree+$2$)}-list assignment of $G$. 
If $G$ contains \ref{t18}, then let $G':=G-\{w,u,t\}$.
If $G'\cong C_5$, then $G\not\cong C_5$ and $\Delta(G)\leq 3$. It follows that $G$ is {\rm PCF}-{\rm (degree+$2$)}-choosable by Theorem \ref{thm:cubic}\ref{two}, a contradiction. Thus, $G'\not\cong C_5$ and it has a {\rm PCF}-$L$-coloring $\phi$.
Let $\{c_x,c_y,c_v\}:=\{\phi(x),\phi(y),\phi(v)\}$.
Let $\alpha \in U_\phi(x,G')$ and $\beta \in U_\phi(y,G')$. 
We now extend $\phi$ to $G$ by coloring $u$ with $\gamma\in L(u) \setminus \{c_x,c_y,c_v,\beta\}$, $w$ with $\xi\in L(w) \setminus \{c_x,\alpha,\gamma\}$, and $t$ with $\zeta\in L(t) \setminus \{c_v,\gamma,c_y\}$. 
It is easy to see that the coloring obtained is a {\rm PCF}-$L$-coloring.
\end{proof}

\begin{lemma}
    The configuration \ref{t19} is {\rm (degree+$2$)}-reducible within $\mathcal{G}_\ell$ for every $\ell \geq 4$.
\end{lemma}

\begin{proof}
Let $G\in \mathcal{G}_\ell$ be a {\rm PCF}-{\rm (degree+$2$)}-critical graph and let $L$ be a {\rm (degree+$2$)}-list assignment of $G$. 
If $G$ contains \ref{t19}, then let $G':=G-\{w,u,t\}$. 
If $G'\cong C_5$, then $G\not\cong C_5$ and $\Delta(G)\leq 3$. It follows that $G$ is {\rm PCF}-{\rm (degree+$2$)}-choosable by Theorem \ref{thm:cubic}\ref{two}, a contradiction. Thus, $G'\not\cong C_5$ and it has a {\rm PCF}-$L$-coloring $\phi$.
Let $\{c_x,c_y,c_v,c_z\}:=\{\phi(x),\phi(y),\phi(v),\phi(z)\}$.
Let $\alpha \in U_\phi(x,G')$ and $\beta\in U_\phi(y,G')$.
We color $u$ with $\gamma\in L(u) \setminus \{c_x,c_y,c_v,\beta\}$, $w$ with $\xi\in L(w) \setminus \{c_x,\alpha,\gamma\}$, and 
$t$ with  $\zeta \in L(t) \setminus \{\gamma,c_v,c_y\}$. Since $v$ has degree two in $G'$, $c_x\neq c_z$. This along with the fact that $\zeta\ne c_y$ give that both $v$ and $u$ have {\rm PCF}-colors and thus $\phi$ has already extended to a {\rm PCF}-$L$-coloring of $G$. 
\end{proof}


\begin{lemma}
    The configuration \ref{t20} is {\rm (degree+$2$)}-reducible within $\mathcal{G}_4$.
\end{lemma}

\begin{proof}
Let $G\in \mathcal{G}_\ell$ be a {\rm PCF}-{\rm (degree+$2$)}-critical graph and let $L$ be a (degree+2)-list assignment of $G$. 
If $G$ contains \ref{t20}, then $G':=G-\{u,v\}$ has a {\rm PCF}-$L$-coloring $\phi$ as $G' \not\cong C_5$. 
Let $\{c_p,c_t,c_y,c_w,c_x,c_b\}:=\{\phi(p),\phi(t),\phi(y),\phi(w),\phi(x),\phi(b)\}$.
Choose $\alpha \in U_\phi(x,G')$, $\beta \in U_\phi(w,G')$, and $\gamma \in U_\phi(y,G')$. 

If we can color $u$ with a color in $L(u) \setminus \{c_x,c_w,\alpha\}$ and $v$ with a color in 
$L(v) \setminus \{c_x,c_w,\beta\}$ so that $u$ and $v$ receive distinct colors, then $\phi$ is extended to $G$ already. Hence, 
$L(u) =\{c_x,c_w,\alpha,\zeta\}$ and $L(v)=\{c_x,c_w,\beta,\zeta\}$ for some color $\zeta$. It follows $\alpha,\beta \notin \{c_x,c_w\}$. Since $c_w\neq c_y$, we have $c_y=\alpha$, $c_t=c_w$, and $\{c_b,c_p\}=\{c_x,\beta\}$. If we can  recolor $p$ with a color $\xi$ in $L(p) \setminus \{c_w,c_y,\gamma,c_p\}$, then 
$\phi$ can be extended to $G$ by coloring $u,v$ with $\zeta,\beta$, respectively. Otherwise, $L(p)=\{c_w,c_y,\gamma,c_p\}$. 
We can recolor $w$ with a color $\eta$ in $L(w) \setminus \{c_w, c_x, c_b, \alpha, \tau_w(b)\}$, $p$ with $c_w$, and color $u,v$ with $\alpha,c_w$, respectively. This extends $\phi$. 
\end{proof}




\begin{lemma}
    The configuration \ref{t21} is {\rm (degree+$2$)}-reducible within $\mathcal{G}_4$.
\end{lemma}

\begin{proof}
Let $G\in \mathcal{G}_\ell$ be a {\rm PCF}-{\rm (degree+$2$)}-critical graph and let $L$ be a {\rm (degree+$2$)}-list assignment of $G$. 
If $G$ contains \ref{t21}, then $G':=G-\{u,v,w\}$ has a {\rm PCF}-$L$-coloring $\phi$ as $G' \not\cong C_5$.
Let $\{c_x,c_y,c_t,c_b,c_p,c_q\}:=\{\phi(x),\phi(y),\phi(t),\phi(b),\phi(p),\phi(q)\}$.
Let $\alpha \in U_\phi(x,G')$, $\beta \in U_\phi(y,G')$, and $\gamma \in U_\phi(q,G')$.

By Lemma \ref{lem:5-cycle(4-2-2-2-4)}, $\phi$ can be extended to $G$ unless
there are colors $\zeta_1,\zeta_2,\xi_1,\xi_2$ such that $U_\phi(x, G')=\{\alpha\}=\{c_y\}$, $U_\phi(y, G')=\{\beta\}=\{c_x\}$, $\zeta_1=c_t=c_q$, $\zeta_2=c_b=c_p$, 
$L(u)=L(v)=L(w)=\{\alpha,\beta, \xi_1, \xi_2\}$, $L(x)=\{\alpha,\beta,\zeta_1,\tau_x(t),\tau_x(q),\xi_i\}$ for some $i=1,2$, and $L(y)=\{\alpha,\beta,\zeta_2,\tau_y(b),\tau_y(p),\xi_i\}$ for some $i=1,2$. By symmetry, we assume, without loss of generality, that $\xi_1\in L(y)$.
If there exists a color $\eta$ in $L(p) \setminus \{c_y,c_q,\gamma,c_p\}$, 
then $\phi$ can be extended to $G$ by recoloring $p$ with $\eta$ and coloring $u,v,w$ with $\xi_1,\xi_2,\beta$, respectively. Otherwise, we have $L(p)=\{c_y,c_q,\gamma,c_p\}$, in which case $\phi$ can be extended to $G$ by recoloring $p,y$ with $\beta,\xi_1$ and coloring $u,v,w$ with $\xi_2,\alpha,\beta$, respectively.
\end{proof}



\begin{lemma}
    The configurations \ref{t22}, \ref{t23}, and \ref{t24} are {\rm (degree+$2$)}-reducible within $\mathcal{G}_4$.
\end{lemma}

\begin{proof}
This is an immediate corollary of Lemma \ref{Reduce-F1}, as each of the mentioned configurations contains the configuration referred to by \ref{f1}.
\end{proof}

\begin{lemma}
    The configurations \ref{t25}, \ref{t26}, and \ref{t27} are {\rm (degree+$2$)}-reducible within $\mathcal{G}_4$.
\end{lemma}

\begin{proof}
This is an immediate corollary of Lemma \ref{Reduce-F2}, as each of the mentioned configurations contains the configuration referred to by \ref{f2}.
\end{proof}

\begin{lemma}
    The configuration \ref{t28} is {\rm (degree+$2$)}-reducible within $\mathcal{G}_4$.
\end{lemma}

\begin{proof}
Let $G\in \mathcal{G}_\ell$ be a {\rm PCF}-{\rm (degree+$2$)}-critical graph and let $L$ be a {\rm (degree+$2$)}-list assignment of $G$. 
If $G$ contains \ref{t28}, then $G':=G-\{u,w,v\}$ has a {\rm PCF}-$L$-coloring $\phi$ as $G' \not\cong C_5$.
Let $\{c_x,c_y,c_h,c_p\}:=\{\phi(x),\phi(y),\phi(h),\phi(p)\}$.
By Lemma \ref{lem:5-cycle(4-2-2-2-4)}, $\phi$ can be extended to $G$ unless
there are colors $\zeta_1,\zeta_2,\xi_1,\xi_2$ such that $U_\phi(x, G')=\{\alpha\}=\{c_y\}$, $U_\phi(y, G')=\{\beta\}=\{c_x\}$, $\zeta_1=c_h=c_p=\zeta_2$, 
$L(u)=L(v)=L(w)=\{\alpha,\beta, \xi_1, \xi_2\}$, $L(x)=\{\alpha,\beta,\zeta_1,\tau_x(h),\tau_x(p),\xi_i\}$ for some $i=1,2$, and $L(y)=\{\alpha,\beta,\zeta_2,\tau_y(h),\tau_y(p),\xi_i\}$ for some $i=1,2$. 
However, since $3\leq d(h)\leq 4$, $\tau_x(h)={\rm NULL}$ or $\tau_x(h)=\alpha$, a contradiction.
\end{proof}


\begin{lemma}
    The configurations \ref{t29} and \ref{t30} are {\rm (degree+$2$)}-reducible within $\mathcal{G}_4$.
\end{lemma}

\begin{proof}
Let $G\in \mathcal{G}_\ell$ be a {\rm PCF}-{\rm (degree+$2$)}-critical graph and let $L$ be a {\rm (degree+$2$)}-list assignment of $G$. 
If $G$ contains \ref{t29} or \ref{t30}, then $G':=G-\{u,w,v\}$ has a {\rm PCF}-$L$-coloring $\phi$ as $G' \not\cong C_5$.
Let $\{c_p,c_q\}:=\{\phi(p),\phi(q)\}$. Since $c_p\neq c_q$ in either case, $\phi$ can be extended to $G$ by Lemma \ref{lem:5-cycle(4-2-2-2-4)}.
\end{proof}

\begin{lemma}
    The configurations \ref{t31}, \ref{t32}, and \ref{t33} are {\rm (degree+$2$)}-reducible within $\mathcal{G}_4$.
\end{lemma}

\begin{proof}
Note that $G':=G-\{u,v\} \not\cong C_5$.
This is an immediate corollary of Lemma \ref{Reduce-F3}, as each of the mentioned configurations contains the configuration referred to by \ref{f3}.
\end{proof}

\begin{lemma}
    The configuration \ref{t34} is {\rm (degree+$2$)}-reducible within $\mathcal{G}_\ell$ for every $\ell\geq 4$.
\end{lemma}

\begin{proof}
Let $G\in \mathcal{G}_\ell$ be a {\rm PCF}-{\rm (degree+$2$)}-critical graph and let $L$ be a {\rm (degree+$2$)}-list assignment of $G$. 
If $G$ contains \ref{t34}, then $G':=G-\{u,v,w\}$ has a {\rm PCF}-$L$-coloring $\phi$ as $G' \not\cong C_5$.
Let $\{c_t,c_x,c_y,c_h,c_l,c_p,c_q\}:=\{\phi(t),\phi(x),\phi(y),\phi(h),\phi(l),\phi(p),\phi(q)\}$. 

By Lemma \ref{lem:2-2-2}, $\phi$ can be extended to $G$ unless $U_\phi(x, G')=\{\alpha\}$, $U_\phi(y, G')=\{\beta\}$, and there are two colors $\xi_1$ and $\xi_2$ such that $L(u)=\{c_x, \alpha, \xi_1, \xi_2\}$, $L(v)=\{c_x, c_y, \xi_1, \xi_2\}$, and $L(w)=\{c_y, \beta, \xi_1, \xi_2\}$. 
If $c_t=\alpha$ (the case $c_t=\beta$ is symmetry), then $c_l=c_h$.
If $\alpha\neq \beta$, then $\{c_p,c_q\}=\{\alpha,\beta\}$. 
We recolor  $t$ with  $\zeta\in L(t)\setminus \{c_x,c_y,\alpha\}$.  
If $\zeta=c_l$, then $\phi$ can be extended to $G$ by coloring $u,v,w$ with $\alpha,\xi_1,\beta$; 
and if $\zeta\not=c_l$ (assuming, by symmetry, that $\zeta\neq \xi_2$), then $\phi$ can be extended to $G$ by coloring $u,v,w$ with $\alpha,\xi_1,\xi_2$, respectively. 
If $\alpha=\beta$, then $c_p=c_q$ and $c_h=c_l$. We recolor $t$ with $\zeta' \in L(t) \setminus \{c_x,c_y,c_t\}$.
If $\zeta' \neq \xi_1$, then $\phi$ can be extended to $G$ by coloring $u,v,w$ with $\xi_1,\xi_2,c_t$; and if $\zeta'=\xi_1$, then $\phi$ can be extended to $G$ by coloring $u,v,w$ with $\xi_2,\xi_1,c_t$, respectively.

If $c_t\neq \alpha,\beta$, then $\{c_l,c_h\}=\{c_t,\alpha\}$ and $\{c_p,c_q\}=\{c_t,\beta\}$. Assume, by symmetry, that $\xi_2\neq c_t$. We recolor $t$ with  $\zeta\in L(t)\setminus \{c_x,c_y,c_t\}$.
If $\zeta=\alpha$, then $\phi$ can be extended to $G$ by coloring $u,v,w$ with $\alpha,\xi_1,\xi_2$; 
if $\zeta=\beta$, then $\phi$ can be extended to $G$ by coloring $u,v,w$ with $\xi_2,\xi_1,\beta$;
and if $\zeta\not=\alpha,\beta$, then $\phi$ can be extended to $G$ by coloring $u,v,w$ with $\alpha,\xi_1,\beta$, respectively.
\end{proof}



\begin{lemma}
    The configuration \ref{t35} is {\rm (degree+$2$)}-reducible within $\mathcal{G}_\ell$ for every $\ell\geq 4$.
\end{lemma}

\begin{proof}
Let $G\in \mathcal{G}_\ell$ be a {\rm PCF}-{\rm (degree+$2$)}-critical graph and let $L$ be a {\rm (degree+$2$)}-list assignment of $G$.
If $G$ contains \ref{t35}, then let $G':=G-\{u,v,p,q\}$.
If $G'\cong C_5$, then $G\not\cong C_5$ and $G$ is connected outerplanar graph. It follows that $G$ is {\rm PCF}-(degree+$2$)-choosable by Theorem \ref{thm-OP}, a contradiction. Thus, $G'\not\cong C_5$ and it has a {\rm PCF}-$L$-coloring $\phi$.
This $\phi$ can be quickly extended to $G$ by Lemma \ref{lem:2-2-4-2-2} as $|U_\phi(x, G')|=2$.
\end{proof}



\begin{lemma} \label{lem:t36}
   The configuration \ref{t36} is  {\rm (degree+$2$)}-reducible within $\mathcal{G}_\ell$ for every $\ell\geq 5$.
\end{lemma}

\begin{proof}
Let $G\in\mathcal{G}_\ell$ be a {\rm PCF}-{\rm (degree+$2$)}-critical graph and let $L$ be a {\rm (degree+$2$)}-list assignment of $G$. 
If $G$ contains \ref{t36}, then let $G':=G-\{h,u,v,p,q,l\}$.
Since $G'\not\cong C_5$, it has a {\rm PCF}-$L$-coloring $\phi$.
Let $\{c_x,c_y,c_w,c_a,c_b,c_d\}:=\{\phi(x),\phi(y),\phi(w),\phi(a),\phi(b),\phi(d)\}$. 
Let $\alpha \in U_\phi(x,G')$, $\beta \in U_\phi(y,G')$, and $\gamma \in U_\phi(w,G')$.

If we can extend $\phi$ to a PCF-$L$-coloring $\varphi$ of $G'':=G-\{p,q,l\}$ so that $|U_\varphi(w,G'')|\neq 1$, then $\varphi$ can be extended to $G$ by Lemma \ref{lem:2-2-2}. This implies $|\{c_a,c_b,c_d\}|=2$, because otherwise, we have
$|\{c_a,c_b,c_d\}|=3$ and that $\phi$ can be extended to the desired 
$\varphi$
by Lemma \ref{lem:2-2-2}. So we assume  $c_a=c_b$ and $c_d=\gamma$.
Now, we color $v$ with a color $\xi_1\in L(v)\setminus \{c_w,\gamma\}$, and $u$ with a color $\xi_2\in L(u)\setminus \{c_x,c_w,\xi_1\}$. If $L(h)\neq \{\alpha,c_x,\xi_1,\xi_2\}$, then we color $h$ with a color in $L(h)\setminus \{\alpha,c_x,\xi_1,\xi_2\}$. This results in the desired $\varphi$ and we are done. Thus, $L(h)=\{\alpha,c_x,\xi_1,\xi_2\}$. Next, if 
$L(u)\neq \{c_x,c_w,\xi_1,\xi_2\}$, then we color $h$ with $\xi_2$ and recolor $u$ with a color in 
$L(u)\setminus \{c_x,c_w,\xi_1,\xi_2\}$; and if $L(v)\neq \{c_w,\gamma,\xi_1,\xi_2\}$, then we color $h$ with $\xi_1$ and recolor $v$ with a color in $L(v)\setminus \{c_w,\gamma,\xi_1,\xi_2\}$, and in either case we obtain the desired $\varphi$ and we are done. Thus, we have
$L(u)= \{c_x,c_w,\xi_1,\xi_2\}$ and $L(v)=\{c_w,\gamma,\xi_1,\xi_2\}$.

By symmetry, there are two colors $\zeta_1$ and $\zeta_2$ such that $L(p)=\{c_w,\gamma,\zeta_1,\zeta_2\}$, 
$L(q)=\{c_y,c_w,\zeta_1,\zeta_2\}$, 
and $L(l)=\{\beta,c_y,\zeta_1,\zeta_2\}$.
At this stage, we color $h,u,v,p,q,l$ with $\xi_1,\xi_2,\gamma,c_w,\zeta_1,\zeta_2$, respectively, and recolor $w$ with a color in 
$L(w) \setminus \{c_w,c_a,\gamma,\tau_w(a),\tau_w(b),\tau_w(d)\}$. This give a PCF-$L$ coloring of $G$, where the PCF-color of $w$ is exactly $c_w$.
\end{proof}

\section{Finishing the proofs}\label{sec3}

In Subsection \ref{subsection:Ti}, we have completed the proof of the following.

\begin{lemma}\label{lem:Configurations}
Each configuration $T_i$ with $i\in [35]$ is $(\text{degree}+2)$-reducible within $\mathcal{G}_4$.  
In particular, $T_i$ with $i \in [5,8] \cup [17,19] \cup \{1,2,11,12,15,34,35\}$ is $(\text{degree}+2)$-reducible within $\mathcal{G}_\ell$ for every $\ell \geq 4$.
\end{lemma}

Now, we complete the proofs of the first three main theorems.


\begin{proof}[Proof of Theorem \ref{thm-SP}]
We define $\mathcal{G}_4$ as the class of connected $K_4$-minor-free graphs with maximum degree at most $4$ for this proof.  
Suppose, for contradiction, that this theorem is incorrect. Then there would exist a $\mathrm{PCF}\text{-}(\text{degree}+2)$-critical graph within $\mathcal{G}_4$. By Lemma~\ref{lem:Configurations}, $G$ contains no $T_i$ for any $i \in [12]$. This contradicts Lemma~\ref{lem:sp-new}, since $G$ remains a $K_4$-minor-free graph with maximum degree at most $4$.
\end{proof}


\begin{proof} [Proof of Theorem \ref{thm-O1P}]
We define $\mathcal{G}_4$ as the class of connected outer-1-planar graphs with maximum degree at most $4$.  
Assume, to derive a contradiction, that the theorem fails. Then $\mathcal{G}_4$ would contain a $\mathrm{PCF}\text{-}(\text{degree}+2)$-critical graph. However, by Lemma~\ref{lem:Configurations}, any such graph $G$ cannot contain any $T_i$ for $i \in [35]$. This directly contradicts Lemma~\ref{lem:o1p-new}, as $G$ by construction remains an outer-1-planar graph with maximum degree at most $4$.
\end{proof}


\begin{proof}[Proof of Theorem \ref{thm-g}]
Let $G$ be a planar graph with girth at least $12$. If $\Delta(G)\leq 3$, then $G$ is 
$\mathrm{PCF}\text{-}(\text{degree}+2)$-choosable by Theorem \ref{thm:cubic}\ref{two}. Thus, we assume $\Delta(G)\leq \ell$ and $\ell \geq 4$.
We define $\mathcal{G}_\ell$ as the class of planar graphs with girth at least $12$ and maximum degree at most $\ell$ for this proof.  
Suppose, for contradiction, that this theorem is incorrect. Then there would exist a $\mathrm{PCF}\text{-}(\text{degree}+2)$-critical graph within $\mathcal{G}_\ell$. By Lemmas~\ref{lem:t36} and \ref{lem:Configurations}, $G$ contains no $T_i$ for any $i \in \{1,5,8,11,12,36\}$. This contradicts Lemma~\ref{lem:planar12}, since $G$ remains a planar graph with girth at least $12$ and maximum degree at most $\ell$.
\end{proof}

For the proof of Theorem \ref{thm-O1P-degree+3}, we need the following lemma of Li and Zhang \cite{zbMATH07465224}.

\begin{lemma}\label{lem:o1p-original-full}
\cite[Theorems 3 and 4]{zbMATH07465224}
Let $G$ be an  outer-$1$-planar graph. If $G$ does not contains \ref{t1}, \ref{t3}, or \ref{t14}, then it contains
\begin{enumerate}[label=$X_{\arabic*}$]\setlength{\itemsep}{-3pt}
     \item \label{X1} a triangle $uvw$ with $d(u) = 2$, or
     \item \label{X2} an edge $uv$ with $d(u) = 2$ and $d(v) \leq 3$.
\end{enumerate}
\end{lemma}

\begin{proof}[Proof of Theorem \ref{thm-O1P-degree+3}]
Assume, for the sake of contradiction, that the theorem does not hold. Then, there exists an outer-1-planar graph $G$ that is not $\mathrm{PCF}\text{-}(\text{degree}+3)$-choosable, yet every proper subgraph of $G$ is $\mathrm{PCF}\text{-}(\text{degree}+3)$-choosable. 
By Lemma~\ref{lem:o1p-original-full}, $G$ must contain one of the following configurations: \ref{t1}, \ref{t3}, \ref{t14}, \ref{X1}, or \ref{X2}. But, each of these configurations can be reduced in a way that leads to a contradiction, as detailed below.

Let $L$ be a {\rm (degree+$3$)}-list assignment of $G$. By assumption, $G'$ has a {\rm PCF}-$L$-coloring $\phi$.

\begin{itemize}\setlength{\itemsep}{-3pt}
    \item If $G$ has \ref{t1}, then let $G':=G-\{u\}$ and $\{c_v\}:=\{\phi(v)\}$.
This coloring can be extended to $G$ by coloring $u$ with a color in
$L(u) \setminus \{\phi(v), \alpha\}$, where $\alpha\in U_\phi(v,G')$.

    \item If $G$ has \ref{t3}, then let $G':=G-\{v\}$ and $\{c_x,c_y\}:=\{\phi(x),\phi(y)\}$. Since $\phi(x)\neq \phi(y)$, $\phi$ can be extended to $G$ by coloring $v$ with a color in
$L(v) \setminus \{\phi(x),\phi(y),\alpha,\beta\}$, where $\alpha\in U_\phi(x,G')$ and $\beta\in U_\phi(y,G')$.

    \item If $G$ has \ref{t14}, then let $G':=G-\{v\}$ and $\{c_x,c_y,c_u\}:=\{\phi(x),\phi(y),\phi(u)\}$.
Since $\phi(x)\neq \phi(y)$, $\phi$ can be extended to $G$ by coloring $v$ with a color in
$L(v) \setminus \{\phi(x),\phi(y),\phi(u),\alpha,\beta\}$, where $\alpha\in U_\phi(x,G')$ and $\beta \in U_\phi(y,G')$.

    \item If $G$ has \ref{X1}, then let $G':=G-\{u\}$ and $\{c_v,c_w\}:=\{\phi(v),\phi(w)\}$. Since $\phi(v)\neq \phi(w)$, $\phi$ can be extended to $G$ by coloring $u$ with a color in
$L(u) \setminus \{\phi(v),\phi(w),\alpha,\beta\}$, where $\alpha\in U_\phi(v,G')$ and $\beta\in U_\phi(w,G')$.

    \item If $G$ has \ref{X2}, then let $G':=G-\{u,v\}$ and $\{c_p,c_q,c_w\}:=\{\phi(p),\phi(q),\phi(w)\}$. 
Choose $\alpha \in U_\phi(p,G')$, $\beta \in U_\phi(q,G')$, and $\gamma \in U_\phi(w,G')$. 
Assume $N(u)=\{v,p\}$. If $d(v)=2$, then assume $N(v)=\{u,q\}$. 
Now, $\phi$ can be extended to $G$ by coloring $u$ with $\xi \in
L(u) \setminus \{\phi(p),\alpha\}$, and $v$ with a color in
$L(v) \setminus \{\xi,\phi(p),\phi(q),\beta\}$. 
If $d(v)=3$, then assume $N(v)=\{u,q,w\}$. 
Now, $\phi$ can be extended to $G$ by coloring $v$ with $\zeta\in
L(v) \setminus \{\phi(p),\phi(q),\phi(w),\beta,\gamma\}$, and $u$ with a color in
$L(u) \setminus \{\zeta,\phi(p),\alpha\}$.
\end{itemize}

Since $C_5$ is not properly conflict-free (degree+$2$)-choosable, the bound on (degree+$3$) in Theorem \ref{thm-O1P-degree+3} is tight.
\end{proof}

\section{Concluding remarks and open problems}

In 2023, Fabrici, Lu{\v{z}}ar, Rindo{\v{s}}ov{\'a}, and Sot{\'a}k \cite{FABRICI202380} proved that every planar graph is proper conflict-free 8-colorable and every outerplanar graph is proper conflict-free 5-colorable. Caro, Petru{\v{s}}evski, and {\v{S}}krekovski \cite{CARO2023113221} also showed that every planar graph with girth at least $7$ is proper conflict-free 6-colorable. Independently, we obtain the following list-version enhancements.

\begin{thm} \label{thm:pcf6-pl}
    Every planar graph with girth at least $12$ is proper conflict-free 6-choosable.
\end{thm}

\begin{thm} \label{thm:pcf6-o1p}
    Every outer-$1$-planar graph is proper conflict-free 6-choosable.
\end{thm}

\noindent We outline the proof strategy as follows.  Let $G$ be either a planar graph with girth at least $12$ or an outer-1-planar graph that is not $\mathrm{PCF}\text{-}6$-choosable, while every proper subgraph of $G$ is $\mathrm{PCF}\text{-}6$-choosable.  By employing a similar reasoning process, one can demonstrate that $G$ does not contain any of the configurations labeled as \ref{t1}, \ref{t3}, \ref{t5}, \ref{t8}, \ref{t11}, \ref{t12}, \ref{t14}, \ref{t36}, \ref{X1}, or \ref{X2}.  Combining this finding with Lemmas \ref{lem:planar12} and \ref{lem:o1p-original-full}, we arrive at Theorems \ref{thm:pcf6-pl} and \ref{thm:pcf6-o1p}, respectively. The detailed proofs are left to the interested readers (note that only the proof of the reducibility of \ref{t36} is slightly different from the proof of Lemma \ref{lem:t36}, but now the proof would be much easier as we get a longer list for any 2-vertex and thus we do not need to recolor the $5$-vertex in this configuration).

On the other hand, by combining the main results presented in Theorems \ref{thm-SP}, \ref{thm-O1P}, \ref{thm-g}, and \ref{thm-O1P-degree+3} with Conjecture~\ref{conj-C}, we are led to consider the following problems.
Of course, the conclusion of Conjecture~\ref{conj-C} implies positive answers to these problems; however, they are of independent interest.

\begin{pblm}
Is every connected $K_4$-minor-free graph (or outer-1-planar graph) other than $C_5$ proper conflict-free {\rm (degree+$2$)}-choosable? 
\end{pblm}

\begin{pblm}
Is every planar graph with girth at least $6$ proper conflict-free {\rm (degree+$2$)}-choosable? 
\end{pblm}

It is worth noting again that every $d$-degenerate graph is proper conflict-free {\rm (degree+$d$+1)}-choosable, as proved by Kashima, \v{S}krekovski, and Xu \cite{arXiv:2509.12560}.  
Given a graph $G$ with maximum degree $\Delta$, it is clearly $\Delta$-degenerate, and thus proper conflict-free {\rm (degree+$\Delta$+1)}-choosable.  
On the other hand, we believe that the straightforward greedy algorithm provided by Cranston and Liu \cite{zbMATH07959405}, which yields a proper conflict-free $(2\Delta + 1)$-coloring of $G$, also implies that $G$ is proper conflict-free $(2\Delta + 1)$-choosable.  
Combining these findings, we pose the following question:  

\begin{pblm}  
Is every graph proper conflict-free ${\rm (2\cdot degree + 1)}$-choosable?  
\end{pblm}

To conclude this paper, we introduce a generalized concept. For two non-negative integers $k$ and $\ell$, a graph $G$ is said to be \emph{proper conflict-free $(\ell \cdot \text{degree} + k)$-choosable} if, for any list assignment $L$ of $G$ satisfying $|L(v)| = \ell \cdot d(v) + k$ for all vertices $v \in V(G)$, there exists a proper conflict-free $L$-coloring $\phi$ of $G$.  Towards resolving Conjecture~\ref{conj-C}, we pose the following problem of independent interest:  

\begin{pblm}  
For any $\varepsilon > 0$, does there exist $\Delta_0 = \Delta_0(\varepsilon)$ such that every graph with maximum degree $\Delta \geq \Delta_0$ is proper conflict-free $\left((1 + \varepsilon) \cdot {\rm degree}\right)$-choosable?  
\end{pblm}

\bibliography{ref}
\bibliographystyle{abbrv}

\end{document}